\documentclass[a4paper]{amsart}
\usepackage{amssymb}
\usepackage{amscd}
\usepackage{amsthm}
\usepackage{amsmath}
\usepackage{latexsym}
\usepackage[all]{xy}
\usepackage{lscape} 
\usepackage{fullpage} 
\usepackage{hyperref}

\usepackage{color}

\numberwithin{equation}{section}

\theoremstyle{plain}
\newtheorem*{theorem*}{Theorem}
\newtheorem{theorem}{Theorem}

\newtheorem{corollary}[theorem]{Corollary}
\newtheorem{lemma}[theorem]{Lemma}
\newtheorem{proposition}[theorem]{Proposition}

\theoremstyle{definition}
\newtheorem{definition}[theorem]{Definition}
\newtheorem{example}[theorem]{Example}

\theoremstyle{remark}

\newtheorem{remark}[theorem]{Remark}

\numberwithin{theorem}{section}

\newcommand{\Div}[1]{\mbox{\rm{Div}}(#1)}
\newcommand{\PDiv}[1]{\mbox{\rm{PDiv}}(#1)}
\newcommand{\HDiv}[1]{\mbox{\rm{HDiv}}(#1)}
\newcommand{\HPDiv}[1]{\mbox{\rm{HPDiv}}(#1)}
\newcommand{\Cl}[1]{\mbox{\rm{Cl}}(#1)}

\newcommand{\HCl}[1]{\mbox{\rm{HCl}}(#1)}
\newcommand{\Chow}[1]{\Cl{#1}}
\newcommand{\Pic}[1]{\mbox{\rm{Pic}}(#1)}
\newcommand{\HPic}[1]{\mbox{\rm{HPic}}(#1)}
\newcommand{\C}[1]{\mbox{\rm{C}}(#1)}
\newcommand{\HC}[1]{\mbox{\rm{HC}}(#1)}

\newcommand{\Spec}[1]{\operatorname{Spec}(#1)}

\newcommand{\ProjSch}[1]{\operatorname{Proj}(#1)}

\newcommand{\End}{\mbox{\rm{End\,}}}

\newcommand{\Ass}[1]{\mbox{\rm{Ass}} \,#1}

\newcommand{\Loc}{\mbox{\rm{Loc\,}}}
\newcommand{\Hom}[3]{\mbox{\rm{Hom}}_{#1}(#2,#3)}
\newcommand{\Ext}[4]{\mbox{\rm{Ext}}^{#1}_{#2}(#3,#4)}
\newcommand{\Tor}[4]{\mbox{\rm{Tor}}_{#1}^{#2}(#3,#4)}

\newcommand{\ModA}{\mbox{\rm{Mod-}A}}
\newcommand{\ModB}{\mbox{\rm{Mod-}B}}

\newcommand{\rk}[1]{\operatorname{rank}#1}

\newcommand{\htt}[1]{\operatorname{ht}#1}

\newcommand{\supp}[1]{\mbox{\rm{supp}} \,#1}  
\newcommand{\Ker}[1]{\mbox{\rm{Ker}}\,#1}
\newcommand{\Coker}{\mbox{\rm{Coker}}}
\newcommand{\Img}{\mbox{\rm{Im}}}
\newcommand{\di}{\mbox{\rm{div}}}

\newcommand{\m}{\mathfrak{m}}
\newcommand{\n}{\mathfrak{n}}

\newcommand{\p}{\mathfrak{p}}
\newcommand{\q}{\mathfrak{q}}

\newcommand{\Acal}{\ensuremath{\mathcal{A}}}

\newcommand{\Fcal}{\ensuremath{\mathcal{F}}}

\newcommand{\Dcal}{\ensuremath{\mathcal{D}}}
\newcommand{\Scal}{\ensuremath{\mathcal{S}}}

\newcommand{\Tcal}{\ensuremath{\mathcal{T}}}

\newcommand{\Lcal}{\ensuremath{\mathcal{L}}}

\newcommand{\Bcal}{\ensuremath{\mathcal{B}}}
\newcommand{\Ycal}{\ensuremath{\mathcal{Y}}}
\newcommand{\Xcal}{\ensuremath{\mathcal{X}}}

\newcommand{\Z}{\mathbb{Z}}

\newcommand{\la}{\longrightarrow}

\newcommand{\V}{V}

\newcommand{\PP}{\mathbb{P}}
\newcommand{\Proj}{\operatorname{Proj}}
\newcommand{\QQ}{\mathbb{Q}}

\begin{document}

\title{Flat ring epimorphisms and universal localisations of commutative rings}
\author{Lidia Angeleri H\"ugel, Frederik Marks, Jan {\v S}{\v{t}}ov{\'{\i}}{\v{c}}ek, Ryo Takahashi, Jorge Vit\'oria}
\address{Lidia Angeleri H\"ugel, Dipartimento di Informatica - Settore di Matematica, Universit\`a degli Studi di Verona, Strada le Grazie 15 - Ca' Vignal, I-37134 Verona, Italy} 
\email{lidia.angeleri@univr.it}
\address{Frederik Marks, Institut f\"ur Algebra und Zahlentheorie, Universit\"at Stuttgart, Pfaffenwaldring 57, 70569 Stuttgart, Germany}
\email{marks@mathematik.uni-stuttgart.de}
\address{Jan {\v S}{\v{t}}ov{\'{\i}}{\v{c}}ek, Charles University in Prague, Faculty of Mathematics and Physics, Department of Algebra, Sokolovsk\'a 83, 186 75 Praha, Czech Republic}
\email{stovicek@karlin.mff.cuni.cz}
\urladdr{http://www.karlin.mff.cuni.cz/~stovicek/}
\address{Ryo Takahashi, Graduate School of Mathematics, Nagoya University, Furocho, Chikusaku, Nagoya, Aichi 464--8602, Japan}
\email{takahashi@math.nagoya-u.ac.jp}
\urladdr{http://www.math.nagoya-u.ac.jp/~takahashi/}
\address{Jorge Vit\'oria, Department of Mathematics, City, University of London, Northampton Square, London EC1V 0HB, UK / Dipartimento di Matematica e Informatica, Universit\'a degli Studi di Cagliari, Palazzo delle Scienze, Via Ospedale, 72, 09124, Cagliari, Italy}
\email{jorge.vitoria@unica.it}
\urladdr{https://sites.google.com/view/jorgevitoria/}
\begin{abstract}
We study different types of localisations of a commutative noetherian ring. More precisely, we provide criteria to decide: (a) if a given flat ring epimorphism is a universal localisation in the sense of Cohn and Schofield; and (b) when such universal localisations are classical rings of fractions. In order to find such criteria, we use the  theory of support and we analyse the specialisation closed subset associated to a flat ring epimorphism. In case the underlying ring is locally factorial or of Krull dimension one, we show that all flat ring epimorphisms are universal localisations. Moreover, it turns out that an answer to the question of when universal localisations are classical depends on the structure of the Picard group. We furthermore discuss the case of normal rings, for which the divisor class group plays an essential role to decide if a given flat ring epimorphism is a universal localisation. Finally, we explore several (counter)examples which highlight the necessity of our assumptions.
\end{abstract}
\thanks{The authors acknowledge support from the Program �Ricerca di Base 2015� of the University of Verona. Lidia Angeleri H\"ugel was partly supported by Istituto Nazionale di Alta Matematica INdAM-GNSAGA. Jan \v{S}\v{t}ov\'{\i}\v{c}ek acknowledges support from the grant GA~\v{C}R 17-23112S from the Czech Science Foundation. Ryo Takahashi was partly supported by JSPS Grants-in-Aid for Scientific Research 16H03923 and 16K05098. Jorge Vit\'oria acknowledges support from the Department of Computer Sciences of the University of Verona in the earlier part of this project, as well as from the Engineering and Physical Sciences Research Council of the United Kingdom, grant number EP/N016505/1, in the later part of this project.}
\date{\today}
\maketitle

\setcounter{tocdepth}{1}
\tableofcontents


\section{Introduction}
Flat epimorphisms of commutative noetherian rings is a very classical topic which was studied in the 1960's in the context of algebraic geometry (see \cite{G,L}). Since then, powerful ring theoretic and homological techniques such as universal localisation~\cite{Scho} and triangulated localisation~\cite{N} were developed. With their help, we are now able to obtain a rather complete picture of how flat epimorphisms look like and how they are controlled by Picard groups, divisor class groups and local cohomology.

\smallskip

In commutative algebra, different types of localisations can be classified geometrically by studying certain subsets of the prime spectrum. More precisely, for a commutative noetherian ring $A$, it was shown in \cite{N} that the assignment of support yields a bijection between localising subcategories of the derived category $\Dcal(A)$ and arbitrary subsets of $\Spec A$. This bijection restricts to a correspondence between smashing subcategories and specialisation closed subsets, which again are in bijection with hereditary torsion classes of the module category $\ModA$. Classically, these torsion classes are linked to the so-called Gabriel localisations and hence to flat ring epimorphisms starting in $A$ (see \cite{G,St}).

One of the key observations here is that a different type of ring theoretical localisation, so-called universal localisation, is very closely connected to flat ring epimorphisms in the context of commutative noetherian rings. In fact, on many occasions flat epimorphisms and universal localisations coincide.

Universal localisations were first introduced in \cite{Scho} in order to extend classical localisation theory in a meaningful way to non-commutative rings. Universal localisations are constructed by adjoining universally the inverses to a set $\Sigma$ of maps between finitely generated projective $A$-modules. As a result, we obtain a ring epimorphism $A\longrightarrow A_\Sigma$ satisfying $\Tor{1}{A}{A_\Sigma}{A_\Sigma}=0$. Universal localisations have proved to be useful in many different areas of mathematics (see \cite{AS,NR,R}), but the concept still remains rather mysterious and classification results are only available in a few cases.

In the context of commutative noetherian rings, however, universal localisation acquires a very concrete geometric meaning: instead of making universally invertible an element of the ring, that is, a section of the structure sheaf of the corresponding affine scheme, we rather make invertible a section in a general invertible sheaf.

Following this point of view, we prove the following result.
The first assertion does not require $A$ to be noetherian, but the second does; otherwise, there is a counterexample.
For the sake of completeness, we remind the reader that if $A\longrightarrow B$ is a ring epimorphism and $A$ is commutative, so is $B$. If, moreover, $B$ is flat over $A$ and $A$ is noetherian, so is $B$ (see \cite{S,L}).

\begin{theorem}[Corollary \ref{r1} and Proposition \ref{Tor1}]\label{r2}
Let $A$ be a commutative ring.
\begin{enumerate}
\item
Any universal localisation of $A$ is a flat ring epimorphism.
\item
A ring epimorphism $A\longrightarrow B$ with $A$ noetherian is flat if and only if $\Tor{1}{A}{B}{B}=0$.
\end{enumerate}
\end{theorem}

Flat ring epimorphisms of commutative noetherian rings are related to the notion of coherent subsets of $\Spec A$, which has been introduced by Krause (see \cite{K}).
More precisely, it turns out that flat ring epimorphisms $A\longrightarrow B$ bijectively correspond to the smashing subcategories in $\Dcal(A)$ whose orthogonal class is closed under cohomologies, which in the terminology of~\cite{K} precisely means that the corresponding specialisation closed subset of $\Spec A$ has coherent complement.
We characterise such specialisation closed subsets in terms of vanishing of local cohomology and we obtain the following classification of flat ring epimorphisms.

\begin{theorem}[Theorem \ref{flat epis} and Corollary \ref{coherent complement}]
Let $A$ be a commutative noetherian ring.
Let $U$ be a generalisation closed subset of $\Spec A$.
Then $U$ is coherent if and only if $H_V^i(A)=0$ for all $i>1$, where $V=\Spec A\setminus U$, and the latter condition implies that the minimal primes in $V$ are of height $\le1$.
In particular, there is a one-to-one correspondence
$$
\left\{
\begin{matrix}
\text{equivalence classes of}\\
\text{flat ring epimorphisms from $A$}
\end{matrix}
\right\}
\overset{\rm 1-1}{\longleftrightarrow}
\left\{
\begin{matrix}
\text{generalisation closed}\\
\text{coherent subsets of $\Spec A$}
\end{matrix}
\right\}.
$$
\end{theorem}

Note that the concept of coherent subset of the spectrum, originally introduced in module-theoretic terms, now, acquires a geometric interpretation via local cohomology.

\smallskip

Recall from Theorem \ref{r2} that universal localisations always yield flat ring epimorphisms. Clearly, classical localisations at a multiplicative set $S$ are always universal localisations with respect to the set of maps $A\longrightarrow A$ given by multiplication with an element from $S$. The reverse implications, however, generally do not hold -- we obtain the following hierarchy of inclusions
\[
\left\{
\begin{matrix}
\textrm{equivalence classes}\\
\textrm{of classical}\\
\textrm{localisations} \\
f\colon A \longrightarrow B
\end{matrix}
\right\}
\overset{\Phi}\subseteq
\left\{
\begin{matrix}
\textrm{equivalence classes}\\
\textrm{of universal}\\
\textrm{localisations}\\
f\colon A \longrightarrow B
\end{matrix}
\right\}
\overset{\Psi}\subseteq
\left\{
\begin{matrix}
\textrm{equivalence classes}\\
\textrm{of flat ring}\\
\textrm{epimorphisms}\\
f\colon A \longrightarrow B
\end{matrix}
\right\}
\overset{\Theta}{\hookrightarrow}
\left\{
\begin{matrix}
\textrm{specialisation closed}\\
\textrm{subsets } V\subseteq\Spec A\\
\textrm{with minimal primes}\\
\textrm{of height } \le 1
\end{matrix}
\right\}	
\]
where $\Theta$ is injective and sends $f\colon A\longrightarrow B$ to the complement of the image of $f^\flat\colon \Spec B \longrightarrow \Spec A$.

In order to distinguish universal localisations from classical rings of fractions (inclusion $\Phi$), we will analyse the structure of the Picard group $\Pic A$ formed by isoclasses of invertible (hence projective) modules. In order to decide if a given flat ring epimorphism is a universal localisation (inclusion $\Psi$), we will study the corresponding specialisation closed subset (injection $\Theta$). The condition on the height of the minimal primes suggests to consider Weil divisors on $A$, i.e.~formal integer combinations of prime ideals of height one, together with the associated divisor class group $\Chow A$, which is the quotient group obtained by factoring out the principal Weil divisors. This group will be of particular relevance in the case of normal rings, where it contains the Picard group as a subgroup. We obtain the following main results.
\smallskip

\begin{theorem}\label{main thm - divisor groups}
Let $A$ be a commutative noetherian ring. 
\begin{enumerate}
\item
{\rm(Theorems \ref{Pic torsion} and \ref{locally factorial})}
If $\Pic{A}$ is torsion, the inclusion $\Phi$ is an equality, i.e. every universal localisation is a classical ring of fractions. Moreover, the converse holds true if $A$ is locally factorial.
\item
{\rm(Theorems \ref{Krull dim one}, \ref{locally factorial} and \ref{thm normal domain})}
Suppose that $A$ satisfies one of the following conditions:
\begin{enumerate}
\item $A$ has Krull dimension at most one, or
\item $A$ is locally factorial, or (more generally)
\item $A$ is normal and $\Chow A/\Pic A$ is torsion.
\end{enumerate}
Then the inclusions $\Psi$ and $\Theta$ are bijections. In particular, flat ring epimorphisms are universal localisations.
\item
{\rm(Proposition \ref{prop criterion for coherence HLVT} and Theorem \ref{thm normal domain})}
If $A$ is normal, has Krull dimension two and is a G-ring (e.g.\ a flat epimorphic image of a finitely generated algebra over a field or of a complete local ring), then $\Theta$ is a bijection.
In this case, $\Psi$ is an equality if and only if $\Chow A/\Pic A$ is torsion, and both $\Phi$ and $\Psi$ are equalities if and only if $\Chow A$ is torsion.
\end{enumerate}
\end{theorem}

We will show, through various examples, that most assumptions in the theorem cannot be omitted.

In case $A$ is normal, it follows that every flat ring epimorphism is a classical localisation whenever $\Chow A$ is torsion (and so are $\Pic A$ and $\Chow A/\Pic A$). This was essentially already observed in \cite[IV, Proposition 4.5]{L}. In case $A$ is even locally factorial,  we obtain that all flat ring epimorphisms are classical localisations if and only if $\Pic A\cong\Chow A$ is torsion. For Dedekind domains this goes back to \cite[Remark on p. 47]{S} and \cite[IV, Proposition 4.6]{L}.
 
Finally, observe that statement $(2)$, case $(\textrm{c})$ above is not presented in its most general form. In fact, for a specialisation closed subset $V$ it is enough to know that all minimal prime ideals are of height one and torsion in $\Chow A/\Pic A$ to guarantee that $V$ is associated with a universal localisation $f$. Conversely, whenever we know that $V$ has a unique minimal prime ideal $\p$, then the existence of a corresponding universal localisation $f$ already implies that $\p$ is torsion in $\Chow A/\Pic A$. 

\smallskip

\noindent \textbf{Structure of the paper:}  We begin in Section 2 with some preliminaries on ring theoretical and categorical localisations. In Section 3, we focus on commutative rings and recall some facts on supports and local cohomology, followed by a discussion of several groups associated to a commutative ring: the Picard group, the Cartier divisors, the Weil divisors and the divisor class group. Section 4 explores some general facts on the relation between flat ring epimorphisms and universal localisations of commutative noetherian rings, while Section 5 is devoted to developing the theory behind Theorem~\ref{main thm - divisor groups} in this introduction. Finally, in Section 6 we explore some examples showing that our assertions of Section 5 cannot be naively generalised.
\smallskip


\noindent \textbf{Notation.} Let $A$ be a ring (in most occasions, commutative noetherian), and let $\ModA$ be the category of all right $A$-modules. We denote by $\Dcal(A)$ the unbounded derived category of $\ModA$.  All subcategories considered are strict, i.e. closed under isomorphic images. For a subcategory $\Bcal$ of a category $\Acal$, we denote by $\Bcal^{\perp}$ the full subcategory formed by the objects $X$ of $\Acal$ such that $\Hom{\Acal}{B}{X}=0$ for all $B$ in $\Bcal$. 
When $A$ is a commutative ring with prime spectrum $\Spec A$, and $\mathfrak p$ is in $\Spec A$, we denote by $k(\p)=A_\p/\p A_\p$ the residue field of $A$ at $\p$.


\section{Preliminaries: Localisations}

\subsection{Localisations of rings} A {\bf ring epimorphism} $f\colon A\longrightarrow B$ is a ring homomorphism which is an epimorphism in the category of
rings. It is well-known that $f\colon A\longrightarrow B$ is a ring epimorphism if and only if $B\otimes_A\Coker{f}=0$ or equivalently, if and only if the associated restriction functor $f_\ast\colon \ModB\longrightarrow \ModA$ is fully faithful. We say that two ring epimorphisms $f_1\colon A\longrightarrow B_1$ and $f_2\colon A\longrightarrow B_2$ are \textbf{equivalent} if $f_1=\psi f_2$ for a ring isomorphism $\psi$. The equivalence classes of this relation are called {\bf epiclasses}.

Recall that a full subcategory $\Ycal$ of $\ModA$ is {\bf bireflective} if the inclusion functor $i\colon \Ycal\longrightarrow \ModA$ has both a left adjoint and a right adjoint functor, or equivalently, $\Ycal$ is closed under products, coproducts, kernels, and cokernels.
\begin{theorem}\cite[Theorem 1.2]{GdP}\cite[Theorem 4.8]{Scho}\label{bireflective}
There is a bijection between
\begin{enumerate}
\item epiclasses of ring epimorphisms $A\longrightarrow B$ (with $\Tor{1}{A}{B}{B}=0$),
\item bireflective (extension-closed) subcategories of $\ModA$,
\end{enumerate}
sending $f\colon A\longrightarrow B$ to the essential image of the restriction functor \mbox{$f_\ast\colon \ModB\longrightarrow \ModA$}.
\end{theorem}

A ring epimorphism $f\colon A\longrightarrow B$ is said to be a {\bf (left) flat epimorphism} if $B$ is a flat (left) $A$-module.
More generally, $f$ is  a  {\bf homological ring epimorphism} if $\mbox{Tor} _i^A(B,B)=0$ for all $i>0$, or equivalently, if the restriction functor $f_\ast\colon \Dcal(B)\longrightarrow\Dcal(A)$ induced by $f$ is fully faithful (\cite[Theorem 4.4]{GL}).
Here is an important tool to construct ring epimorphisms.

\begin{theorem} \cite[Theorem~4.1]{Scho}\label{uniloc} 
Let  $\Sigma$ be a set of
morphisms between finitely generated projective right $A$-modules.
Then there are a ring $A_\Sigma$ and a morphism of rings
$f_\Sigma\colon A\longrightarrow A_\Sigma$, called the \textbf{universal localisation of $A$ at
$\Sigma$}, such that
\begin{enumerate}
\item $f_\Sigma$ is \textbf{$\Sigma$-inverting,} i.e. if
$\alpha\colon P\longrightarrow Q$ belongs to  $\Sigma$, then
$\alpha\otimes_A 1_{A_\Sigma}\colon P\otimes_A A_\Sigma\longrightarrow
Q\otimes_A A_\Sigma$ is an isomorphism of right
$A_\Sigma$-modules, and
\item $f_\Sigma$ is \textbf{universal
$\Sigma$-inverting}, i.e. for any $\Sigma$-inverting ring homomorphism $\psi\colon A\longrightarrow B$, there is a unique ring homomorphism $\bar{\psi}\colon A_\Sigma\longrightarrow B$ such that $\bar{\psi}f_\Sigma=\psi$.
\end{enumerate}
Moreover, the homomorphism $f_\Sigma$ is a ring epimorphism and $\Tor1A{A_\Sigma}{A_\Sigma}=0$.
\end{theorem}

Easy examples of universal localisations can be obtained from the classical construction of a ring of fractions of a commutative ring $A$ with denominators in a multiplicative set $S$ (which we will often refer to as \textbf{classical localisations}). Indeed, these are universal localisations at the set of $A$-module endomorphisms of $A$ given by multiplication by the elements of $S$. Note, however, that universal localisations of an arbitrary ring $A$ in general are not flat, and not even  homological ring epimorphisms. We will, however, prove in Section \ref{Section ring epi} that universal localisations of commutative rings yield flat ring epimorphisms.

 
\subsection{Categorical localisations} \label{subsection catg loc}
A pair of full subcategories $(\Tcal,\Fcal)$ of $\ModA$ is said to be a \textbf{torsion pair} if $\Hom{A}{\Tcal}{\Fcal}=0$ and, for every $A$-module $M$, there is a short exact sequence 
$$0\longrightarrow T \longrightarrow M\longrightarrow F\longrightarrow 0$$
with $T$ in $\Tcal$ and $F$ in $\Fcal$. It turns out that the modules $T$ and $F$ in such a sequence depend functorially on $M$, and the endofunctor of $\ModA$ sending $M$ to $T$ is said to be the \textbf{torsion radical} associated to $\Tcal$. The subcategory $\Tcal$ is said to be a \textbf{torsion class} and $\Fcal$ is said to be a \textbf{torsionfree class}. Such classes in $\ModA$ can be characterised by closure conditions, namely a subcategory is a torsion (respectively, torsionfree) class if and only if it is closed under extensions, coproducts and epimorphic images (respectively, extensions, products and submodules). We refer to \cite[Chapter VI]{St} for details. 

A torsion class $\Tcal$ in $\ModA$ is said to be {\bf hereditary} if it is closed under submodules, or equivalently, if the corresponding torsionfree class $\mathcal{F}:=\Tcal^\perp$ is closed under injective envelopes. A hereditary torsion class $\Tcal$ is also a  \textbf{Serre subcategory} (i.e. it is closed under extensions, submodules and epimorphic images) and, therefore, it yields an (exact) quotient functor of abelian categories $i^*\colon\ModA \longrightarrow \ModA/\Tcal$. The inclusion functor $j_!\colon\mathcal{T}\longrightarrow\ModA$ and the quotient functor $i^*\colon \ModA\longrightarrow \ModA/\Tcal$ induce a \textbf{localising sequence of abelian categories}
\begin{equation}\nonumber
(\mathcal{L}_\mathcal{T})\ \ \ \ \ \ \ \ \ \xymatrix@C=0.5cm{
\ModA/\mathcal{T} \ar[rrr]^{i_*} &&& \ModA\ar[rrr]^{j^*}  \ar @/_1.5pc/[lll]_{i^*}  &&& \mathcal{T}
\ar @/_1.5pc/[lll]_{j_!} } 
\end{equation}
i.e. $(i^*,i_*)$ and $(j_!,j^*)$ are adjoint pairs, $j_!$ and $i_*$ are fully faithful and $\Ker(i^*)=\Img(j_!)$. It follows that $\Img(i_*)$ can be identified with $\Tcal^\perp\cap \Ker{\Ext{1}{A}{\Tcal}{-}}$ and that the composition $j_!j^*$ is the torsion radical associated to  $\Tcal$.
A well-known theorem states that every flat ring epimorphism arises in this way. 

\begin{theorem}\label{perfect loc}
If $f\colon A\longrightarrow B$ is left flat ring epimorphism, then $\Tcal_f:=\Ker(-\otimes_AB)$ is a hereditary torsion class and there is a localising sequence of the form
\begin{equation}\nonumber
(\mathcal{L}_{\mathcal{T}_f})\ \ \ \ \ \ \ \ \ \xymatrix@C=0.5cm{
\ModB \ar[rrr]^{f_\ast} &&& \ModA\ar[rrr]^{j^*}  \ar @/_1.5pc/[lll]_{-\otimes_AB}  &&& \mathcal{T}_f
\ar @/_1.5pc/[lll]_{j_!} } 
\end{equation}
where $f_\ast$ also admits a right adjoint $i^!=\Hom{A}{B}{-}$.
Moreover, given a hereditary torsion class $\Tcal$, the following are equivalent.
\begin{enumerate}
\item There is a left flat ring epimorphism $f\colon A\longrightarrow B$ such that $\Tcal=\Tcal_f$.
\item The functor $i_*\colon \ModA/\mathcal{T}\longrightarrow \ModA$ in the localising sequence $(\mathcal{L}_\mathcal{T})$ admits a right adjoint.
\end{enumerate}
\end{theorem}

\begin{proof}
Let $f\colon A\longrightarrow B$ be a left flat ring epimorphism. It is clear that $\Tcal_f$ is a hereditary torsion class. The fact that there is a localising sequence of the form $(\mathcal{L}_{\Tcal_f})$ follows from \cite[XI, Theorem 2.1]{St} and goes back to \cite{G}.

(1)$\Rightarrow$(2) is clear. Let us prove (2)$\Rightarrow$(1). By Theorem \ref{bireflective}, it follows that there is a ring epimorphism $f\colon A\longrightarrow B$ and an equivalence between $\ModA/\mathcal{T}$ and $\ModB$ such that $i_*$ is identified with $f_\ast$. The flatness of $B$ as a left $A$-module follows from the fact that the quotient functor $i^*\colon \ModA\longrightarrow \ModA/\Tcal$ (which is naturally equivalent to $-\otimes_AB$) is exact.
\end{proof}

A triangulated subcategory of $\Dcal(A)$ is called {\bf localising} if it is closed under coproducts, and a localising subcategory $\Scal$ of $\Dcal(A)$ is called {\bf smashing} whenever the inclusion functor $j_!\colon\Scal\longrightarrow\Dcal(A)$ admits a right adjoint that preserves coproducts. In this case, the inclusion $j_!$ and the quotient functor $i^*\colon \Dcal(A)\longrightarrow \Dcal(A)/\Scal$ induce a so-called \textbf{recollement of triangulated categories}
\begin{equation}\nonumber
\xymatrix@C=0.5cm{
\Dcal(A)/\Scal \ar[rrr]^{i_*} &&& \Dcal(A)\ar[rrr]^{j^*}  \ar @/_1.5pc/[lll]_{i^*}  \ar
 @/^1.5pc/[lll]_{i^!} &&& \Scal
\ar @/_1.5pc/[lll]_{j_!} \ar
 @/^1.5pc/[lll]_{j_*}
 } 
 \end{equation}
i.e. $(i^*,i_*,i^!)$ and $(j_!,j^*,j_*)$ are adjoint triples, $i_*$, $j_!$, and $j_*$ are fully faithful and $\Ker(i^*)=\Img(j_!)$. It follows that $\Img(i_*)$ can be identified with $\Scal^\perp$.

\begin{example}\label{example flat ring epi recollement}
If $f\colon A\longrightarrow B$ is a left flat ring epimorphism, then $\Scal_f:=\Ker(-\otimes_A^{\mathbb{L}}B)$ is a smashing subcategory and there is an associated recollement of the form 
\begin{equation}\nonumber
\xymatrix@C=0.5cm{
\Dcal(B) \ar[rrr]^{f_*} &&& \Dcal(A)\ar[rrr]^{j^*}  \ar @/_1.5pc/[lll]_{-\otimes_A^{\mathbb{L}}B}  \ar
 @/^1.5pc/[lll]_{i^!} &&& \Scal_f
\ar @/_1.5pc/[lll]_{j_!} \ar
 @/^1.5pc/[lll]_{j_*}
 } 
 \end{equation}
where $\Scal_f$ identifies with the smallest localising subcategory of $\Dcal(A)$ containing the cone of $f$ when viewed as a morphism in $\Dcal(A)$. Moreover, we have $i^!=\mathbb{R}\Hom{A}{B}{-}$ (see, for example, \cite{NS1}).
\end{example}

In Section \ref{Section ring epi}, we will show that for a commutative noetherian ring $A$ a smashing subcategory $\Scal$ of $\Dcal(A)$ arises from a flat ring epimorphism in the way above if and only if its orthogonal class $\Scal^\perp$ is closed for cohomologies. This observation will provide a derived analogue of Theorem \ref{perfect loc}.


\section{Preliminaries: Commutative Algebra}\label{commutative algebra}

In this section, we  assume $A$ to be a commutative noetherian ring.

\subsection{Supports and local cohomology}
Recall that a prime ideal $\mathfrak p$ is \textbf{associated} to an $A$-module $M$ if $A/\mathfrak p$ is isomorphic to a submodule of $M$. We denote by $\Ass M$ the set of prime ideals associated to~$M$.
Following \cite{K}, we define the \textbf{support} of a complex  of $A$-modules $X$ as
$$\supp X=\{\mathfrak p\in\Spec A\,\mid\, X\otimes^\mathbb{L}_Ak(\p)\not=0\}.$$
For a module $M\in\ModA$, this definition specialises to
$$\supp M=\{\mathfrak p\in\Spec A\,\mid\, \Tor{\ast}{A}{M}{k(\p)}\not=0\}.$$ 
Equivalently (see \cite[Proposition 2.8 and Remark 2.9]{Foxby} or \cite[Lemma 3.3]{K}), $\supp M$ is given by the prime ideals $\p$ occurring as associated primes of the injective modules appearing in a minimal injective coresolution of $M$.

Recall that the Zariski-closed subsets of $\Spec A$ are of the form $V(I):=\{\p\in\Spec A\mid I\subseteq \p\}$, for an ideal $I$ of $A$. If $I$ is a principal ideal generated by an element $x$ of $A$, we denote $V(I)$ simply by $V(x)$. A subset $V$ of $\Spec A$ is said to be \textbf{specialisation closed} if for any primes $\mathfrak{p}\subseteq \mathfrak{q}$, if $\mathfrak{p}$ lies in $V$, then so does $\mathfrak{q}$. Equivalently, $V$ is specialisation closed if it is a union of Zariski-closed subsets of $\Spec A$. In particular, every specialisation closed subset is the union of the Zariski-closures of the primes $\mathfrak{p}$  which are minimal in $V$. 

Observe that, by Nakayama's lemma, the support of a finitely generated module $M$ coincides with the prime ideals $\mathfrak p$ in $\Spec A$ such that $M\otimes_A A_\p\not=0$ (sometimes referred to as the \textit{classical support of $M$}). In general, the classical support of a module $M$ is the Zariski-closure of $\supp M$, see \cite[Lemma 2.2]{BIK}. This shows that a specialisation closed subset $V$ of $\Spec{A}$ contains $\supp M$ if and only if it contains the classical support of $M$. Since the latter has the same minimal elements as $\Ass M$, it also follows that $\supp M\subseteq V$ if and only  if $\Ass M\subseteq V$.

In some cases, the support of a complex $X$ is  determined by the support of its  cohomologies, that is  $\supp X=\bigcup_{i\in\Z}\supp H^i(X)$. 
For instance, if $X$ is a bounded complex of finitely generated $A$-modules, we can compute $\supp X$ by taking the classical support of its cohomologies.
For a  two-term complex of finitely generated projective $A$-modules  $X\colon  P_{-1}\stackrel{\sigma}{\longrightarrow} P_0$, this means that $\supp X$ consists of the prime ideals $\p$ in $\Spec A$ for which $\sigma_\p=\sigma\otimes_A A_\p$ is not an isomorphism. We will often denote the two-term complex $X$ simply by the map $\sigma$. For example, we will write $\supp \sigma$ instead of $\supp X$.

The notion of support turns out to be fundamental in the understanding of hereditary torsion classes in $\ModA$ and smashing subcategories of $\Dcal(A)$.

\begin{theorem}\label{Gabriel's-Neeman's Classification}
\cite[Ch.VI,\S 5,6]{St}\cite{N}
Let $A$ be a commutative noetherian ring. There are bijections between
\begin{enumerate}
\item hereditary torsion classes in $\ModA$;
\item smashing subcategories of $\Dcal(A)$;
\item specialisation closed subsets of $\Spec A$.
\end{enumerate}
More precisely, the bijections are given by assigning to a hereditary torsion class or to a smashing subcategory its support, which is then specialisation closed in $\Spec A$.
 \end{theorem}
 \begin{remark}\label{remark Gabriel's-Neeman's Classification}
 The bijective correspondence between (1) and (2) in the theorem above can also be made explicit. To a hereditary torsion class $\Tcal$ in $\ModA$ we assign the smallest localising subcategory of $\Dcal(A)$ containing $\Tcal$. In the other direction, to a smashing subcategory $\Scal$ of $\Dcal(A)$ we associate the subcategory $\Scal\cap\ModA$ of $\ModA$.
 \end{remark}

Given a specialisation closed subset $V$, we denote the associated hereditary torsion class by $\Tcal_V$ and the associated smashing subcategory by $\Scal_V$.

\begin{remark}\label{remark Neeman}
It also follows from \cite{N} that $\Scal_V\,^\perp$ coincides with the subcategory of $\Dcal(A)$ formed by the objects with support contained in $\Spec A\setminus V$. Indeed, in \cite{N} it is shown that the assignment of support yields a bijection between localising subcategories of $\Dcal(A)$ and subsets of $\Spec A$. Moreover, the support of a localising subcategory $\mathcal{L}$ can be detected as the subset of $\Spec{A}$ formed by the prime ideals $\mathfrak{p}$ whose residue field $k(\mathfrak{p})$ lies in $\mathcal{L}$. As  shown in \cite[Lemma 3.5]{N}, either $k(\mathfrak{p})$ lies in $\mathcal{L}$ or $k(\mathfrak{p})$ lies in $\mathcal{L}^{\perp}$. Now, if $\mathcal L$ is a smashing subcategory, then $\Lcal^\perp$ is localising and, by the arguments above, the associated subsets of the spectrum are complementary.
\end{remark}

If we consider the recollement
\begin{equation} \label{recollement for S_V}
\xymatrix@C=0.5cm{
\Scal_V\,^\perp \ar[rrr]^{i_*} &&& \Dcal(A)\ar[rrr]^{j^*}  \ar @/_1.5pc/[lll]_{i^*}  \ar
 @/^1.5pc/[lll]_{i^!} &&& \Scal_V
\ar @/_1.5pc/[lll]_{j_!} \ar
 @/^1.5pc/[lll]_{j_*}
 } 
\end{equation}
with $j_!$ and $i_*$ being just the inclusion functors, it turns out that the functor $j_!j^*$ coincides with the right derived functor of the torsion radical $\Gamma_V$ of $\Tcal_V$. Indeed, if $X$ is a homotopically injective complex, we have a (by \cite[Corollary 2.1.5]{BS2} even componentwise split) short exact sequence
\[ 0 \longrightarrow \Gamma_V(X) \longrightarrow X \longrightarrow X/\Gamma_V(X) \longrightarrow 0, \]
which induces a triangle in $\Dcal(A)$. Now it follows from~\cite[Lemma 2.10]{N} that $\Gamma_V(X)\in\Scal_V$ and $X/\Gamma_V(X)\in\Scal_V\,^\perp$, so $\Gamma_V(X)\longrightarrow X$ is an $\Scal_V$-reflection, as claimed.

\begin{example}\label{quick way to see support}
Let $f\colon A\longrightarrow B$ be a  flat ring epimorphism, and let $V$ be the specialisation closed subset corresponding to the smashing subcategory $\Scal_f:=\Ker(-\otimes_A^{\mathbb{L}}B)$ from Example~\ref{example flat ring epi recollement}. It follows from Remark~\ref{remark Neeman} that a  prime ideal $\mathfrak{p}$ belongs to $\Spec A\setminus V$ if and only if  $B\otimes_A k(\mathfrak{p})\cong k(\mathfrak{p})$, and it belongs to $V$ if and only if  $B\otimes_A k(\mathfrak{p})=0$, or equivalently, $B=\mathfrak{p}B$. 

Moreover,  it is shown in \cite[IV, Proposition 2.1]{L} that all ideals $J\subseteq B$ are of the form $J=IB$ where $I=f^{-1}(J)$. Hence the map of spectra $f^\flat\colon \Spec{B} \longrightarrow \Spec{A},\, \mathfrak{q}\mapsto f^{-1}({\mathfrak q})$ is a homeomorphic embedding of $\Spec{B}$ onto $\Spec A\setminus V=\supp B$, see also \cite[IV, Proposition 1.4 and  Corollaire 2.2]{L}.
 \end{example}

\begin{definition}
The \textbf{local cohomology} of an $A$-module $M$ with respect to a specialisation closed subset $V$ is $H_V^i(M):=H^i(\mathbb{R}\Gamma_V M)$, where $\Gamma_V$ is the torsion radical of $\Tcal_V$. We will denote $\Gamma_{V(I)}(M)$ and $H_{V(I)}^i(M)$ simply by $\Gamma_I(M)$ and $H_I^i(M)$, respectively.
\end{definition}

Note that $\Gamma_V(M)=\{x\in M\mid\supp(xA)\subseteq V\}$ and, thus, we may recover from above the classical sheaf-theoretic definition of local cohomology from~\cite[Chapter IV]{Har}. Alternatively, we also have that $H_V^i(M)=H^i(j_!j^*M)$, where $j_!$ and $j_*$ are functors appearing in the recollement of $\Dcal(A)$ induced by $\Scal_V$.

We will also need well-known key properties of local cohomology: the independence of the base and the flat base change. Classically, these were studied for specialisation closed subsets of the form $V(I)$, where $I\subseteq A$ is an ideal.

\begin{proposition} \cite[\S4.2 and \S4.3]{BS2} \label{properties of H_I}
Let $f\colon A\longrightarrow B$ be a homomorphism of commutative noetherian rings and $I\subseteq A$ be an ideal. We denote by $f_*\colon \ModB\longrightarrow\ModA$ the restriction functor.

\begin{enumerate}
\item \textbf{Independence of the base:} If $N$ lies in $\ModB$, then there is a functorial isomorphism $f_*(H^i_{IB}(N)) \cong H^i_I(f_*(N))$ in $\ModA$ for each $i\ge 0$.
\item \textbf{Flat base change:} If $M$ lies in $\ModA$ and $B$ is flat over $A$, then there is a functorial isomorphism $H^i_{IB}(M\otimes_A B) \cong H^i_I(M)\otimes_A B$ in $\ModB$ for each $i\ge 0$.
\end{enumerate}
\end{proposition}

We will now generalise the proposition to arbitrary specialisation closed subsets. First of all,  for each $A$-module $M$ and each specialisation closed subset $V$ of $\Spec{A}$ we have the equality
\[ \Gamma_V(M) = \bigcup_{I\subseteq A, V(I)\subseteq V} \Gamma_I(M). \]
Indeed, if $x$ is an element of $M$ with $\supp{xA}\subseteq V$, then we can write $\supp{xA} = V(I) \subseteq V$ where $I$ is the annihilator of $x$. Next, we note that this union is directed since the union of finitely many Zariski closed subsets of $\Spec{A}$ is again Zariski closed. Now recall that the local cohomology functors of $M$ are computed by applying the torsion radical to the injective resolution of $M$.  Since cohomologies commute with direct limits, we conclude (cf.~\cite[Page 219, Motif D]{Har}) that for all $i\geq 0$
\[ H^i_V(M) = \varinjlim_{I\subseteq A, V(I)\subseteq V} H^i_I({M}). \]
Taking direct limits of the isomorphisms from Proposition~\ref{properties of H_I}, we obtain the following statement.

\begin{corollary} \label{properties of H_V}
Let $f\colon A\longrightarrow B$ be a homomorphism of commutative noetherian rings, and denote by $f^\flat\colon \Spec{B} \longrightarrow \Spec{A}$ the induced map of spectra and by $f_*\colon \ModB\longrightarrow\ModA$ the corresponding restriction functor. Suppose that $V\subseteq \Spec{A}$ is specialisation closed and $W = (f^\flat)^{-1}(V)$.

\begin{enumerate}
\item If $N$ lies in $\ModB$, then there is a functorial isomorphism $f_*(H^i_W(N)) \cong H^i_V(f_*(N))$ in $\ModA$ for each $i\ge 0$.
\item If $M$ lies in $\ModA$ and $B$ is flat over $A$, then there is a functorial isomorphism $H^i_W(M\otimes_A B) \cong H^i_V(M)\otimes_A B$ in $\ModB$ for each $i\ge 0$.
\end{enumerate}\end{corollary}

The following two theorems are concerned with vanishing and non-vanishing of local cohomology. The first one, due to Grothendieck, guarantees the vanishing of local cohomology for degrees larger than the Krull dimension.

\begin{theorem}\cite[Theorem 6.1.2]{BS2}\label{Gr vanishing}
Let $A$ be a commutative noetherian ring of Krull dimension $d$ and $I$ an ideal of $A$. Then we have $H_I^i(A)=0$ for all $i>d$.
\end{theorem}

The second one, known as the \emph{Hartshorne-Lichtenbaum Vanishing Theorem} (or HLVT for short), is a useful tool to detect the non-vanishing of local cohomology when the degree equals the Krull dimension. Recall that for a commutative noetherian local ring $A$ with maximal ideal $\mathfrak{m}$, the completion $\widehat{A}$ of $A$ at $\mathfrak{m}$ (defined as the ring $\varprojlim A/\mathfrak{m}^n$) is again a commutative noetherian local ring of the same Krull dimension.

\begin{theorem}\cite[Theorem 8.2.1]{BS2}\label{HLVT'}
Let $A$ be a local commutative noetherian ring of Krull dimension $d$ and with maximal ideal $\m$. Given a proper ideal $I$ of $A$,  the following statements are equivalent.
\begin{enumerate}
\item $H^d_I(A)= 0$;
\item for every prime ideal $\q$ of $\widehat{A}$ satisfying $\dim\widehat{A}/\q=d$, we have $\dim\widehat{A}/(I\widehat{A}+\q)>0$.
\end{enumerate}
\end{theorem}

\begin{remark}\label{rem Grothendieck}
If $A$ is a local commutative noetherian ring of Krull dimension $d$ and with maximal ideal $\m$, since $\dim\widehat{A}/(\mathfrak{m}\widehat{A}+\q)=0$ for any prime $\q$ of $\widehat{A}$, it follows from the above theorem that $H^d_\mathfrak{m}(A)\neq 0$. This statement is known as \emph{Grothendieck's non-vanishing Theorem} (see \cite[Theorem 6.1.4]{BS2}).

Note also that the primes $\q$ of $\widehat{A}$ satisfying $\dim\widehat{A}/\q=d$ must necessarily be minimal primes of $\widehat{A}$.
\end{remark}


\subsection{Invertible ideals and divisors}
We say that an element of $A$ is \textbf{regular} if it is not a zero-divisor. Let $K$ denote the (classical) localisation of $A$ at all regular elements of $A$, that is, the total ring of fractions of $A$.

\begin{definition}
An $A$-module $M$ is said to be \textbf{invertible} if it is finitely generated and locally free of rank one (i.e. $M_\mathfrak{p}\cong A_\mathfrak{p}$ for all $\mathfrak{p}\in\Spec{A}$).
\end{definition}

Note that, since projectivity can be checked locally, an invertible $A$-module is, by definition, projective. Invertible modules are so called due to the well-known fact that the evalutation map yields an isomorphism $M\otimes_A M^\ast\cong A$, where $M^\ast=\Hom{A}{M}{A}$ (\cite[Theorem 11.6a]{E}). Isoclasses of invertible $A$-modules form an abelian group under the operation $\otimes_A$. This group is called the \textbf{Picard group of $A$} and it is denoted by $\Pic{A}$.

\begin{definition} 
An $A$-module $M$ is said to be a \textbf{fractional ideal of $A$} if it is a finitely generated $A$-submodule of $K$.
\end{definition}

Fractional ideals are so named due to the fact that, once we choose a common denominator $g\in A$ for the generators of such a module $M$, we see that $gM$ can be identified with an actual ideal of $A$, which is isomorphic to $M$ as an $A$-module. 

\begin{theorem}\cite[Theorem 11.6 and Corollary 11.7]{E}
Let $A$ be a commutative noetherian ring and $K$ be its total ring of fractions. Then the following holds.
\begin{enumerate}
\item Every invertible $A$-module is isomorphic to an invertible fractional ideal of $A$.
\item Every invertible fractional ideal of $A$ contains a regular element of $A$.
\item Invertible fractional ideals form an abelian group for the operation of multiplication of $A$-sub\-modules of $K$. The inverse of an invertible fractional ideal $M$ is $M^{-1}:=\{s\in K:sM\subseteq A\}$. This group is the group of \textbf{Cartier divisors of $A$} and is denoted by $\C{A}$.
\item The group $\C{A}$ is generated by the invertible ideals of $A$.
\item The assignment sending an invertible fractional ideal $M$ to its isoclass yields a surjective homomorphism of abelian groups $\C{A}\la \Pic{A}$, whose kernel  is isomorphic to $K^\times/A^\times$.
\end{enumerate}
\end{theorem}

Note that for an invertible fractional ideal, one has that $M^{-1}\cong M^\ast$. It is quite easy to see that $\C{A}$ is indeed generated by the invertible ideals of $A$. Let $I$ be an element of $\C{A}$ and let $g$ be a regular element of $A$ such that $gI\subseteq A$. Clearly, $gA$ is also an invertible ideal of $A$ and $I=(gI)(gA)^{-1}$.
\begin{definition}
The free abelian group with basis formed by the prime ideals of $A$ of height one is denoted by $\Div{A}$ and its elements (i.e. formal integer  combinations of such prime ideals) are called \textbf{(Weil) divisors}. A Weil divisor is said to be \textbf{effective} if it is a non-negative integer combination of prime ideals of height one. Given a Weil divisor $x=\sum n_\mathfrak{p}\mathfrak{p}$, we say that $x$ \textbf{is supported on} the (finite) subset of $\Spec{A}$ formed by the primes  $\mathfrak{p}$ with $n_\mathfrak{p}\neq 0$.
\end{definition}

Let $I$ be an invertible ideal of $A$ and let $\mathfrak{p}\in V(I)$ be a prime ideal  of height one (we write $\rm{ht}(\mathfrak{p})=1$). Since $I$  contains a regular element, $\mathfrak{p}$ corresponds to one of the finitely many minimal primes in $A/I$, and $\dim A_\mathfrak{p}/I_\mathfrak{p}=0$. Hence the $A_{\mathfrak p}$-module $A_\mathfrak{p}/I_\mathfrak{p}$ has finite length   $\ell(A_\mathfrak{p}/I_\mathfrak{p})>0$ for a finite number of primes $\mathfrak{p}$ of height one, and vanishes for the others.

\begin{theorem}\cite[Theorem 11.10]{E}
Let $A$ be a commutative noetherian ring. There is a homomorphism of abelian groups $\di\colon \C{A}\la \Div{A}$ which sends an invertible ideal $I$ of $A$ to 

$$
\di(I):=\sum\limits_{\mathrm{ht}(\mathfrak{p})=1}\ell(A_\mathfrak{p}/I_\mathfrak{p})\cdot\mathfrak{p}.
$$
\end{theorem}

A general formula for $\rm{div}$ can then be derived from the fact that the invertible ideals of $A$ generate $\C{A}$. If $I$ is an invertible fractional ideal, then $$\di(I):=\di(gI)-\di(gA)$$ where $g$ is a regular element of $A$ such that $gI\subseteq A$.

\begin{remark}\label{div and support}
If $I$ is an invertible ideal of $A$ (hence a projective $A$-module) and $\sigma$ denotes its inclusion into $A$, then $\ell(A_\mathfrak{p}/I_\mathfrak{p})$ in the expression above is non-zero precisely on the primes $\mathfrak{p}$ of height one which lie in $\supp{\sigma}$. In other words, we have that $\di(I)$ is supported on $\supp{\sigma}\cap \{\mathfrak{p}\in\Spec{A}:\rm{ht}(\mathfrak{p})=1\}$.
\end{remark}

\begin{definition}
A divisor is said to be \textbf{principal} if it is the image under $\rm{div}$ of a cyclic invertible fractional ideal, that is, it has the form $\di(sA)$ for some $s$ in $K^\times$. We write  $\di(s)$ for short. The \textbf{divisor class group} $\Chow{A}$ of $A$ is defined to be the quotient of $\Div{A}$ by the subgroup $\PDiv{A}$ of principal divisors.
\end{definition}

\begin{remark}
Our reference~\cite[\S11.5]{E} uses a different terminology and notation. The group $\Chow{A}$ is called the codimension-one Chow group and is denoted by $\operatorname{Chow}(A)$ there.
\end{remark}

Note that in particular, for an invertible fractional ideal $I$ and $g$ regular in $A$ such that $gI \subseteq A$, the divisors $\di(I)$ and $\di(gI)$ get identified in $\Chow{A}$. In fact, it is easy to see that the image of $\di(I)$ in $\Chow{A}$ depends only on the isoclass of $I$. Therefore, the composition $\C{A}\la \Div{A}\la \Chow{A}$ factors through a group homomorphism $\overline{\rm{div}}\colon \Pic{A}\la \Chow{A}$, and we get a commutative diagram:
\begin{equation}\label{diagram of divisors}\vcenter{\xymatrix{0\ar[r]&\{sA:s\in K^\times/A^\times\}\ar[r]\ar@{->>}[d]^{\rm{div}}&\C{A}\ar[r]\ar[d]^{\rm{div}}&\Pic{A}\ar[d]^{\overline{\rm{div}}}\ar[r]&0 \\ 0\ar[r]& \PDiv{A}\ar[r] & \Div{A}\ar[r]&\Chow{A}\ar[r]&0}}\end{equation}

In Section \ref{subsection divisors advanced} we will also  discuss variants of groups of divisors for graded rings and their induced projective schemes. 


\section{Flat ring epimorphisms and universal localisations: General facts}\label{Section ring epi}
In this section, we discuss some general facts about ring epimorphisms and universal localisations of commutative rings. Note that if $A$ is a commutative ring and $A\longrightarrow B$ is a ring epimorphism, then also $B$ is commutative (\cite[Corollary 1.2]{S}). Moreover, if $A$ is noetherian and $A\longrightarrow B$ is flat, then also $B$ is noetherian (\cite[IV, Corollary 2.3]{L}).
We begin with a general construction of new ring epimorphisms from given ones.

\begin{lemma}\label{new ring epis}
Let $A$ be an arbitrary ring, $f\colon A\longrightarrow B$ and $g\colon A\longrightarrow C$ be ring epimorphisms with associated bireflective subcategories $\Xcal_B$ and $\Xcal_C$. Denote by $B\sqcup_A C$ (together with the morphisms $h_B\colon B\longrightarrow B\sqcup_A C$ and $h_C\colon C\longrightarrow B\sqcup_A C$) the pushout of $f$ and $g$ in the category of rings. Then $h_B$ and $h_C$ are ring epimorphisms and the bireflective subcategory associated to the composition $h_Bf=h_Cg$ is given by $\Xcal_B\cap \Xcal_C$.
\end{lemma}

\begin{proof}
It follows from the very definition of the pushout that $h_B$ and $h_C$ are ring epimorphisms (and thus so is the composition $h_Bf=h_Cg$). Moreover, an $A$-module $X$ belongs to $\Xcal_B\cap \Xcal_C$ if and only if the $A$-action $\alpha_X\colon A\longrightarrow\End_\Z(X)$ factors through both $f\colon A\longrightarrow B$ and $g\colon A\longrightarrow C$. Using the universal property of the pushout, this is further equivalent to $\alpha_X$ factorising through $A\longrightarrow B\sqcup_A C$ or, in other words, to $X$ lying in the bireflective subcategory $\Xcal_{B\sqcup_A C}$ associated to the composition $h_Bf=h_Cg$.
\end{proof}

Note that in the lemma above, if $A$ is a commutative ring, then the pushout $B\sqcup_A C$ is given by $B\otimes_AC$ (together with the two natural maps $h_B=id_B\otimes_A1_C$ and $h_C=1_B\otimes_Aid_C$).
In particular, given a ring epimorphism $f\colon A\longrightarrow B$ and a prime $\mathfrak{p}$ in $\Spec A$, there is a ring epimorphism $f_\p\colon A_\mathfrak{p}\longrightarrow B_\mathfrak{p}$ and, consequently, also a ring epimorphism $A\longrightarrow B_\mathfrak{p}$ which we shall denote by $\hat{f}_\mathfrak{p}$.

\subsection{Universal localisations}
Let us now turn our attention to universal localisations of a commutative ring $A$. Our first target is to show that these localisations always yield flat ring epimorphisms (independently of $A$ being noetherian or not). In order to do so, we begin with studying universal localisations from a local point of view. We need the following auxiliary results.

\begin{lemma}\label{matrix}
Let $A$ be a commutative ring and let $\sigma\colon A^n\longrightarrow A^m$ be a linear map with associated matrix $M\in M_{m\times n}(A)$.
\begin{enumerate}
\item If $n=m$, then $\sigma$ is an isomorphism if and only if ${\rm det}M$ is invertible in $A$.
\item If $n\not=m$ and $\sigma$ is an isomorphism, then $A=0$.
\end{enumerate}
\end{lemma}
\begin{proof} (1) is well known. (2) follows from the fact that any non-zero commutative ring has the invariant basis number property (i.e. isomorphic finitely generated free modules must have the same rank).
\end{proof}

\begin{lemma}\label{localisations over local rings}
Any universal localisation of a local commutative ring  is a classical localisation.
\end{lemma}
\begin{proof}
This follows from the previous lemma, since finitely generated projective modules over local rings are well-known to be free.
\end{proof}

Given a prime ideal $\p$ in $\Spec A$ and a universal localisation $f_\Sigma\colon A\longrightarrow A_\Sigma$, we can construct two further ring epimorphisms, namely $(f_\Sigma)_\p:A_\p\longrightarrow (A_\Sigma)_\p$ and $(\hat{f}_\Sigma)_\p:A\longrightarrow (A_\Sigma)_\p$. Note that it follows from Lemma \ref{new ring epis} and the very definition of universal localisation that $(f_\Sigma)_\p$ is the universal (and hence classical) localisation of $A_\p$ at the set $\Sigma_\p:=\{A_\p\otimes_A\sigma\mid \sigma\in\Sigma\}$ and that $(\hat{f}_\Sigma)_\p$ is the universal localisation of $A$ at the set $\Sigma\cup\{A\overset{s}{\longrightarrow}A\mid s\in A\setminus\p\}$. Now we have the desired result.

\begin{corollary}\label{r1}
Any universal localisation of a commutative ring is a flat ring epimorphism.
\end{corollary}

\begin{proof}
We can check flatness of $A_\Sigma$ locally. Now it is enough to observe that by Lemma \ref{localisations over local rings} the universal localisation $(f_\Sigma)_\p:A_\p\longrightarrow (A_\Sigma)_\p$ must be classical and, hence, $(A_\Sigma)_\p$ is a flat $A_\p$-module.
\end{proof}


\subsection{Flat ring epimorphisms}
Throughout this subsection, we consider a commutative noetherian ring $A$ and we study flat ring epimorphisms $A\longrightarrow B$. We begin with an easy homological description.

\begin{proposition}\label{Tor1}
Let $A$ be a commutative noetherian ring. Then a ring epimorphism $f\colon A\longrightarrow B$ is flat if and only if $\Tor{1}{A}{B}{B}= 0$. 
\end{proposition}
\begin{proof}
We only prove the if-part as the other implication is trivial. In view of~\cite[Theorem 7.1]{Mat}, it is sufficient to prove that for each $\q\in\Spec B$ and $\p = f^\flat(\q)$, the induced ring epimorphism $A_\p\longrightarrow B_\q$ is flat. Thus, we may without loss of generality assume that $A$ and $B$ are local and $f$ is a local homomorphism (i.e.\ $f(\m)\subseteq \n$, where $\m\subseteq A$ and $\n\subseteq B$ are the maximal ideals). We will prove that under these assumptions, $f$ is even an isomorphism.

To this end, let $\Xcal_B$ denote the essential image of the restriction functor in $\ModA$. By assumption, $\Xcal_B$ is an exact abelian subcategory closed under products and extensions. Since $\overline{f}\colon A/\m \to B/\n$ is a non-zero ring epimorphism and $A/\m$ is a field, $\overline{f}$ is an isomorphism by~\cite[Corollary 1.2]{S} (see also \cite[IV, Corollaire 1.3]{L}). In particular, $A/\m$ lies in $\Xcal_B$. It then follows inductively that $A/\m^n$ is in $\Xcal_B$ for each $n\ge1$ and also that $\widehat{A}=\varprojlim A/\m^n$ belongs to $\Xcal_B$.
Thus, $f\otimes_A \widehat{A}$ is an isomorphism. Since $A$ is noetherian, $\widehat{A}$ is a faithfully flat $A$-module by \cite[Theorems 7.2 and 8.8]{Mat} and the conclusion follows since $\widehat{A}\otimes_A-$ reflects isomorphisms.
\end{proof}

\begin{remark}
Note that the fact that $A$ is noetherian is needed only in the very last step of the proof to infer that $\widehat{A}\otimes_A-$ reflects isomorphisms. This may fail for non-noetherian commutative rings as may the conclusion of Proposition~\ref{Tor1} (see \cite[Theorem 7.2]{BS}).
\end{remark}

\begin{remark}
Alternatively, one can deduce Proposition~\ref{Tor1} as a consequence of \cite[Lemma 3.5]{K}, which implies that every injective $B$-module is injective as an $A$-module. It then suffices to apply this to the character module $B^+:=\Hom{\mathbb{Z}}{B}{\mathbb{Q}/\mathbb{Z}}$.
\end{remark}

\begin{corollary}
Let $A$ be a commutative noetherian ring. Then every bireflective extension-closed subcategory of $\ModA$ is equivalent to $\ModA/\Tcal$ for some hereditary torsion class $\Tcal$.
\end{corollary}
\begin{proof}
This follows from combining the proposition above with Theorems \ref{bireflective} and \ref{perfect loc}.
\end{proof}

Next, we are aiming for a restricted version of Theorem \ref{Gabriel's-Neeman's Classification} identifying the smashing subcategories of $\Dcal(A)$ and the specialisation closed subsets of $\Spec A$ that correspond to flat ring epimorphisms.

\begin{theorem}\label{flat epis}
Let $A$ be a commutative noetherian ring and let $V$ be a specialisation closed subset of $\Spec A$ with associated smashing subcategory $\Scal_V$. The following statements are equivalent.
\begin{enumerate}
\item $\Scal_V\,^\perp$ is closed under taking cohomologies.
\item $H_V^{k}(A)=0$ for all $k>1$.
\item There is a flat ring epimorphism $f\colon A\longrightarrow B$ such that $\Scal_V=\Ker(-\otimes_A^{\mathbb{L}}B)$ and $\Scal_V\,^\perp\cong\Dcal(B)$.
\item The modules supported in $\Spec A\setminus V$ form an exact abelian extension-closed subcategory of $\ModA$.
\end{enumerate}
In particular, if $V$ satisfies the conditions above, then the minimal primes of $V$ have height zero or one.
\end{theorem}

\begin{proof}
Consider the recollement induced by the smashing subcategory $\Scal_V$ and fix the notation for its functors as in~\eqref{recollement for S_V}.
The units and counits of the adjunctions in the recollement induce a triangle
$$j_!j^*A\la A\la i_*i^*A\la j_!j^*A[1].$$
The long exact sequence of cohomologies associated to this triangle tells us that,  for all $k\geq 1$, $$H^k(i_*i^*A)\cong H^{k+1}(j_!j^*A)=H^{k+1}_V(A).$$

(1)$\Rightarrow$(2): Suppose that $\Scal_V\,^\perp$ is closed under taking cohomologies.  Since $\Tcal_V\subseteq\Scal_V$ by Remark~\ref{remark Gabriel's-Neeman's Classification}, we infer that $H^{k+1}_V(A)$ lies in both $\Scal_V$ and its orthogonal $\Scal_V\,^\perp$ for all $k\geq 1$, that is, $H^k_V(A)=0$ for all $k>1$.

(2)$\Rightarrow$(3): Recall from \cite{NS1} that every recollement of $\Dcal(A)$ is equivalent to one induced by a homological epimorphism of differential graded algebras $f\colon A\longrightarrow B$, where $B$ is identified, as an object in $\Dcal(A)$, with $i_*i^*A$. 
Let us compute the cohomologies of $i_*i^*A$. By definition of local cohomology  $H^k(j_!j^*A)=H_V^k(A)=0$ for all $k<0$, and $H^0(j_!j^*A)=\Gamma_V(A)$ embeds in $A$. It follows from the triangle above that $H^k(i_*i^*A) = 0$ for all ${k}<0$. Now the assumption (2) guarantees that $i_*i^*A$ is an $A$-module and  $f$ is a homological epimorphism of rings up to quasi-isomorphism. Then $f$ is  a flat ring epimorphism by Proposition \ref{Tor1}, and we can invoke Example \ref{example flat ring epi recollement}.

(3)$\Rightarrow$(4): By Remark \ref{remark Neeman}, the complexes supported in $\Spec A\setminus V$ are precisely those in $\Scal_V\,^\perp$. Hence, the $A$-modules supported in $\Spec A\setminus V$ are those in $\Scal_V\,^\perp\cap \ModA$. By assumption (3), these are precisely the $A$-modules in the essential image of the restriction functor $f_*$, which naturally form an exact abelian and extension-closed subcategory of $\ModA$.

(4)$\Rightarrow$(1): This follows from \cite[Theorem 1.1]{K}.

Finally, let $\mathfrak{p}$ be a minimal prime in $V$. In particular, the preimage of $V$ under $\Spec{A_\p} \longrightarrow \Spec{A}$ is just $\{\p A_\p\}$. Then, by Corollary~\ref{properties of H_V}(2), we have $(H^k_V(A))_\p \cong H^k_{\p A_\p}(A_\p)$.
Now, if condition (2) is satisfied, Remark~\ref{rem Grothendieck} guarantees that $\mathfrak{p}$ has height at most one.
\end{proof}

In \cite{K}, the subsets $W\subseteq\Spec A$ satisfying that the modules supported in $W$ form an exact abelian and extension-closed subcategory of $\ModA$ are called \textbf{coherent}. Note that the class of modules supported in any subset $W$ is always extension-closed and that every specialisation closed subset $V$ is coherent (the modules supported there form the hereditary torsion class $\Tcal_V$). Theorem~\ref{flat epis} describes when the complement of a specialisation closed subset is coherent in terms of local cohomology. This may be regarded as a geometric interpretation of coherence.

The following corollary provides a restriction of the bijections  in Theorem \ref{Gabriel's-Neeman's Classification}.

\begin{corollary}\label{coherent complement}
Let $A$ be a commutative noetherian ring. There are bijections between
\begin{enumerate}
\item epiclasses of flat ring epimorphisms $A\longrightarrow B$;
\item smashing subcategories $\Scal$ of $\Dcal(A)$ for which $\Scal^\perp$ is closed under taking cohomologies;
\item specialisation closed subsets of $\Spec A$ with coherent complement.
\end{enumerate}
\end{corollary}

\begin{proof}
This is a direct consequence of Theorems~\ref{Gabriel's-Neeman's Classification} and \ref{flat epis}.
\end{proof}

\begin{example}\label{example support for classical loc}
Let $\{x_\lambda\}_{\lambda\in\Lambda}$ be a family of elements of $A$. The set $V=\bigcup_{\lambda\in\Lambda}\V(x_\lambda)$ is a specialisation closed subset of $\Spec A$ with coherent complement that corresponds to the classical localisation of $A$ at the multiplicative subset $S$ generated by the $x_\lambda$ (see also~\cite[\S4]{K}). Indeed, a prime ideal $\p$ lies in $\supp{A_S}$ if and only if $A_S\otimes_A k(\p) \ne 0$ if and only if $\p\cap\{x_\lambda\}_{\lambda\in\Lambda} = \emptyset$.
\end{example}

\begin{example}\label{example non-coherent puncture}
Let $k$ be a field, $A=k[[X,Y]]$ and consider the specialisation closed subset $V=\{\mathfrak{m}\}$ formed by the unique maximal ideal $\mathfrak{m}=(X,Y)$ in $A$. Since $\mathfrak{m}$ has height 2, by Theorem~\ref{flat epis}, $\Spec A\setminus \{\mathfrak{m}\}$ is not coherent (see also \cite{K}) and the recollement induced by the smashing subcategory $\Scal_V$ does not arise from a flat ring epimorphism. 
\end{example}

\begin{example}\label{example two planes}
Let $k$ be a field, and let $A=k[[X,Y,U]]/(XU)$ be a quotient of a formal power series ring.
Then $A$ is a $2$-dimensional complete local hypersurface with maximal ideal $\m=(X,Y,U)$.
The ideal $\p=(X,Y)$ is a prime ideal of $A$ of height $1$, but we claim that $\Spec A\setminus V(\p)$ is not coherent. To see that, consider a hypothetical flat ring epimorphism $f\colon A \longrightarrow B$ with $\supp B = \Spec{A}\setminus V(\p)$. If we put $\overline{A} = A/(U) \cong k[[X,Y]]$ and $\overline{B} = B\otimes_A \overline{A} \cong B/(f(U))$, then $\overline{f}\colon \overline{A} \longrightarrow \overline{B}$ is a flat ring epimorphism and, by Remark~\ref{quick way to see support}, $\supp{\overline{B}} = \Spec{\overline{A}}\setminus \{\overline{\p}\}$, where $\overline{\p} = \p+(U)/(U)$ is the maximal ideal of $\overline{A}$. However, such a flat ring epimorphism cannot exist by Example~\ref{example non-coherent puncture}.
\end{example}


\section{Flat ring epimorphisms and universal localisations: Comparisons}
We have seen that universal localisations yield flat ring epimorphisms over any commutative ring $A$. In this section, we assume furthermore that $A$ is noetherian and we approach the question of when the converse is true.

After some preliminaries, we start in \S\ref{subsection versus} by showing that all universal localisations of $A$ are classical rings of fractions whenever the Picard group $\Pic{A}$  is torsion. We exhibit examples showing that this is not always the case, not even for Dedekind domains. 
 In \S\ref{dimone} we focus on rings of Krull dimension one. For such rings, flat ring epimorphisms coincide with universal localisations.   \S\ref{normal} is devoted to normal rings. We show that the  the divisor class group $\Cl{A}$ and its  quotient $\Cl{A}/\Pic{A}$ by the Picard group determine whether a flat ring epimorphism is a classical ring of fractions or a  universal localisation. In the special case of a locally factorial ring this means that flat ring epimorphisms coincide with universal localisations, and they further coincide with classical rings of fractions if and only if $\Pic{A}$ is torsion. 

We begin with a lemma which will be crucial in this context.

\begin{lemma}\label{suppunivloc}
Let $A$ be a commutative noetherian ring. Let further $f\colon A\longrightarrow B$ be a flat ring epimorphism with associated specialisation closed subset $V$ and let $\Sigma$ be a set of maps between finitely generated projective $A$-modules. Then $f$ is the universal localisation of $A$ at $\Sigma$ if and only if $V=\supp\Sigma$.
\end{lemma}

\begin{proof}
Let $\Scal_V=\Ker(-\otimes_A^{\mathbb{L}}B)$ be the smashing subcategory associated to $V$. It follows from \cite{NR} that $f$ is the universal localisation of $A$ at $\Sigma$ if and only if $\Scal_V$ identifies with $\Loc\Sigma$, the  smallest localising subcategory of $\Dcal(A)$ containing $\Sigma$. But the latter is further equivalent to $V=\supp\Sigma$.
\end{proof}

\begin{remark}\label{suppclassloc-again}
Similarly, we know from Example~\ref{example support for classical loc} that the universal localisation of $A$ at a set  $\Sigma$ coincides with the  classical localisation at a multiplicative subset $S\subset A$ if and only if $\supp\Sigma=\bigcup_{s\in S}V(s)$.
\end{remark}


\subsection{Universal versus classical localisation.}\label{subsection versus}
In order to provide a sufficient condition for a universal localisation to be a classical ring of fractions, we  establish a relation with invertible ideals.

\begin{proposition}\label{reduction}
Let $A$ be a commutative noetherian ring and $\Sigma$ be a set of maps between finitely generated projective $A$-modules. Then the universal localisation of $A$ at $\Sigma$ coincides with the universal localisation of $A$
\begin{itemize}
\item at a set $\Gamma$ of maps of the form $\gamma\colon A\la J_\gamma$, where $J_\gamma$ is an invertible ideal of $A$;
\item at a set $\Gamma^\prime$ of maps of the form $\gamma\colon J_\gamma\la A$, where $J_\gamma$ is an invertible ideal of $A$.
\end{itemize}
Moreover, if $A$ is a domain and $A_\Sigma\not= 0$, then the set $\Gamma^\prime$ above can be chosen to consist of inclusion maps.
\end{proposition}

\begin{proof}
Without loss of generality, we may assume that $A$ is connected and, hence, all finitely generated projective modules have a well-defined rank (i.e. localisations of a projective module at prime ideals have constant rank). First, we claim that if $\Sigma$ contains a map $\sigma\colon P\la Q$ between modules of different rank, then the localisation at $\Sigma$ is zero and, in particular, it can be realised as claimed above. Observe that the difference between the ranks of $P_\mathfrak{p}$ and $Q_\mathfrak{p}$ and the fact that $(A_\Sigma)_\mathfrak{p}\otimes_A\sigma_\mathfrak{p}$ is an isomorphism imply by Lemma \ref{matrix} that $(A_\Sigma)_\mathfrak{p}=0$ for all primes $\mathfrak{p}$. Therefore, we have $A_\Sigma=0$.

Hence, we assume that $\Sigma$ contains only maps between projective modules of the same rank. Given a map $\sigma\colon P\la Q$ between modules of rank $n$, we claim that a homomorphism of commutative rings $f\colon A\longrightarrow B$ inverts $\sigma$ if and only if it inverts the $n$-fold exterior product $\wedge^n\sigma$. Indeed, this is the case because locally at a prime $\mathfrak{p}$, the map $\wedge^n\sigma_\mathfrak{p}$ is the multiplication by the determinant of $\sigma_\mathfrak{p}$, and the matrix $B_\mathfrak{p}\otimes_{A_\mathfrak{p}}\sigma_\mathfrak{p}\in M_{n\times n}(B_\mathfrak{p})$ is invertible if and only if so is its determinant (which is $f_\mathfrak{p}(\det(\sigma_\mathfrak{p}))$).
We conclude that the universal localisation at $\Sigma$ coincides with the universal localisation at $\wedge\Sigma:=\{\wedge^{\rk{P}}\sigma, (\sigma\colon P_\sigma\la Q_\sigma)\in\Sigma\}$.

Without loss of generality we may then assume that $\Sigma$ is a set of maps between invertible $A$-modules. Note that for any $A$-module $M$, and for any $\sigma\colon P\la Q$ in $\Sigma$, the map $M\otimes_A\sigma$ is an isomorphism if and only if $M\otimes_A (\sigma\otimes_A P^\ast)$ is an isomorphism. Hence, we may once again replace the set $\Sigma$ by the set $\Gamma:=\{\sigma\otimes_AP^\ast: (\sigma\colon P\la Q)\in\Sigma\}$. Finally, observe that, up to isomorphism, each map $\sigma\otimes_AP^\ast$ is of the form $A\la J$ for some invertible ideal $J$ of $A$. Analogously, the universal localisation of $A$ at $\Sigma$ identifies with the universal localisation of $A$ at $\Gamma^\prime:=\{Q^\ast\otimes_A\sigma: (\sigma\colon P\la Q)\in\Sigma\}$.

Assume now that $A$ is a domain. Clearly any non-zero map $\gamma\colon  A\longrightarrow J$ to an invertible ideal $J$ is injective since $J$ is a torsion-free module.
Since the image of $\gamma\otimes_AJ^*$ is an invertible ideal of $A$, the statement follows.
\end{proof}

\begin{theorem}\label{Pic torsion}
Let $A$ be a commutative noetherian ring and $\Sigma$ be a set of maps of the form $\sigma\colon A\la J$, with $J$ an invertible ideal of $A$. If the isoclass of each $J$ is a torsion element in $\Pic{A}$, then the universal localisation at $\Sigma$ is a classical ring of fractions. In particular, if $\Pic{A}$ is torsion, then every universal localisation is a classical ring of fractions.
\end{theorem}

\begin{proof}
For a map $\sigma\colon A\la J$ in $\Sigma$, let $n_J$ denote the order of $J$ in $\Pic{A}$. Clearly, $J^{\otimes n_J}\cong A$ and, therefore, $\sigma^{\otimes n_J}\colon A^{\otimes n_J}\longrightarrow J^{\otimes n_J}$ can be identified with the multiplication map by an element of $A$. We claim that the universal localisation of $A$ at $\Sigma$ coincides with the universal localisation at $\{\sigma^{\otimes n_J}: (\sigma\colon A\la J)\in \Sigma\}$ and that, therefore, it is an ordinary ring of fractions.

Analogously to the proof of the previous proposition, consider a ring homomorphism $f\colon A\longrightarrow B$ and let us check that locally at any prime, $f$ inverts a map $\sigma\colon A\la J$ as above if and only if it inverts $\sigma^{\otimes n_J}$. Indeed, at a prime $\mathfrak{p}$ in $\Spec A$,
$B_\mathfrak{p}\otimes_{A_\mathfrak{p}}\sigma_\mathfrak{p}$ is given by multiplication by an element $b$ of $B_\mathfrak{p}$ and, thus, $B_\mathfrak{p}\otimes_{A_\mathfrak{p}}\sigma^{\otimes n_J}_\mathfrak{p}$ is given by multiplication by $b^{n_J}$. Hence, $B_\mathfrak{p}\otimes_{A_\mathfrak{p}}\sigma^{\otimes n_J}_\mathfrak{p}$ is invertible if and only if $B_\mathfrak{p}\otimes_{A_\mathfrak{p}}\sigma_\mathfrak{p}$ is invertible, as wanted.
\end{proof}

We will see below that for locally factorial rings  the converse holds true as well: $\Pic{A}$ is torsion if and only if every universal localisation is a classical ring of fractions.

\begin{example} \label{Dedekind Pic torsion expl}
The ring $A=\mathbb{Z}[\sqrt{-5}]$ is a Dedekind domain (hence, locally factorial) which is not a unique factorisation domain. Its Picard group is $\mathbb{Z}/2\mathbb{Z}$ and, thus, every universal localisation is a classical ring of fractions. The latter property is in fact common to all number fields $A$, i.e.\ the rings of integers of finite field extensions of $\mathbb{Q}$, as they are always Dedekind domains with finite (hence torsion) Picard group, \cite[Ch. 5, Corollary 2]{Marcus}.
\end{example}

\begin{example} \label{Dedekind Pic non-torsion expl}
On the other hand, every abelian group occurs as the ideal class group of some Dedekind domain \cite{Cla}, so there are (non-local) Dedekind domains $A$ admitting a universal localisation $A\longrightarrow B$ which is not classical. Such a universal localisation is by~\cite[p.~47]{S} always the inclusion of $A$ to an intermediate ring $A \subseteq B \subseteq K$, where $K$ is the quotient field of $A$ (see Proposition~\ref{flat epis as generalized quotients} and Theorem~\ref{locally factorial} below for more information).
\end{example}


\subsection{Rings of Krull dimension at most one}\label{dimone}
 Next, we prove a classification result for rings satisfying a certain geometric condition, namely having Krull dimension at most one.
The situation is similar to the results obtained in \cite[Theorem 6.1]{KS} for possibly non-commutative hereditary rings and in \cite[Theorem 6.8]{BS} for commutative semihereditary rings.
Note that for rings of Krull dimension at most one, every subset of $\Spec A$ is coherent (see \cite[Corollary 4.3]{K}) and, by Corollary \ref{coherent complement}, epiclasses of flat ring epimorphisms correspond bijectively to specialisation closed subsets of the spectrum. Given a prime ideal $\mathfrak{p}$ in $\Spec A$, we denote by $\Lambda(\mathfrak{p})$ the set of prime ideals of $A$ which are contained in $\mathfrak{p}$. Note that $\Lambda(\mathfrak{p})$ with the induced topology is homeomorphic to $\Spec {A_\mathfrak{p}}$.
 
\begin{theorem}\label{Krull dim one}
Let $A$ be a commutative noetherian ring of Krull dimension at most one and let $V$ be a specialisation closed subset of $\Spec A$. Then there is a set of maps $\Sigma$ between finitely generated projective $A$-modules such that $\supp(\Sigma)=V$. As a consequence, every flat ring epimorphism is a universal localisation.
\end{theorem}

\begin{proof}
For each minimal element $\mathfrak{p}$ in the set $V$ we find a map $\sigma$ that is supported exactly on $V(\mathfrak{p})$. Since $V$ is the union of such sets, the theorem will then follow. Let $\mathfrak{p}$ be such a prime ideal. By prime avoidance, we can choose $x$ to be an element of $\mathfrak{p}$ that is not contained in any associated prime that does not lie in $V(\mathfrak{p})$, i.e. we can choose $x$ in $\mathfrak{p}\setminus \bigcup\{\mathfrak{q}\,\mid\,{\mathfrak{q}\in\Ass(A)\setminus V(\mathfrak{p})}\}$.

Suppose that $V(x)=V(\mathfrak{p})$. In this case, we show that the map $\sigma\colon  A\longrightarrow A$ given by multiplication by $x$ is as wanted (see also \cite[Remark 4.2(1)]{K}). Indeed, given a prime ideal $\mathfrak{q}$, if $x$ does not lie in $\q$, then $x$ is invertible in $A_\mathfrak{q}$ and, thus, $\sigma_\mathfrak{q}$ is an isomorphism. If $x$ lies in $\q$, then $x$ lies in the unique maximal ideal of $A_\mathfrak{q}$ and, therefore, $\sigma_\mathfrak{q}$ is not surjective (and, thus, not an isomorphism). This shows that the support of $\sigma$ is precisely $V(x)$, which by assumption is $V(\mathfrak{p})$.

Suppose now that $V(x)\supsetneq V(\mathfrak{p})$. The fact that $A$ has Krull dimension at most one guarantees that all ideals in $V(x)\setminus V(\mathfrak{p})$ are both minimal and maximal prime ideals over $A/xA$. In particular, there are only finitely many such primes.
Let $$V(x)\setminus V(\mathfrak{p})=\{\mathfrak{m}_1,...,\mathfrak{m}_n\}.$$ 
We observe that the localisation map $\pi\colon  A/xA\longrightarrow (A/xA)_{\mathfrak{m}_i}$ is surjective. Indeed, this can be checked locally using the standard fact that for a commutative noetherian ring $R$ and $\mathfrak{a}, \mathfrak{b}\in \Spec R$, we have $R_\mathfrak{a}\otimes_R R_\mathfrak{b}=0$ if and only if $\Lambda(\mathfrak{a})\cap\Lambda(\mathfrak{b})=\emptyset$.
In our setting, it then follows that $\pi_\mathfrak{q}=0$ for any $\mathfrak{q}$ in $\Spec{A/xA}\setminus\{\mathfrak{m}_i\}$, and $\pi_{\mathfrak{m}_i}$ is an isomorphism. So $\pi$ is locally surjective, and thus it is surjective. 

As a consequence, also the composition $A\longrightarrow A/xA\longrightarrow (A/xA)_{\mathfrak{m}_i}$ is surjective. We consider its kernel  $I_i$ together with the short exact sequence
$$0\longrightarrow I_i\longrightarrow A\longrightarrow (A/xA)_{\mathfrak{m}_i}\longrightarrow 0.$$
Since localisations are exact, and $(A/xA)_{\mathfrak{m}_i}\cong  A_{\mathfrak{m}_i}/xA_{\mathfrak{m}_i}$, we have $(I_i)_\mathfrak{q}=A_\mathfrak{q}$ whenever $\mathfrak{q}\neq \mathfrak{m}_i$ and $(I_i)_{\mathfrak{m}_i}=xA_{\mathfrak{m}_i}$. An easy consequence is that the ideals $I_1,..., I_n$ are pairwise coprime. Indeed, if $1\leq i\neq j\leq n$, then $(I_i+I_j)_\mathfrak{q}=A_\mathfrak{q}$ for any $\mathfrak{q}$ in $\Spec A$,  showing that $I_i+I_j=A$. Thus, we conclude that the product of ideals $I:=I_1I_2\dots I_n$ equals the intersection $I_1\cap I_2\cap ... \cap I_n$ and contains the element $x$.

We will show that $I$ is a projective $A$-module and that the map $\sigma\colon A\longrightarrow I$ given by multiplication by $x$ is as wanted. To check the projectivity of $I$ we see that $I$ is locally projective. Indeed, since $I_{\mathfrak{q}}=(I_1)_\mathfrak{q}(I_2)_{\mathfrak{q}}...(I_n)_{\mathfrak{q}}$, it follows that $I_{\mathfrak{q}}=A_\mathfrak{q}$ if $\mathfrak{q}$ does not lie in $\{\mathfrak{m}_1,...,\mathfrak{m}_n\}$, and $I_{\mathfrak{m}_i}=xA_{\mathfrak{m}_i}$. Note that since $\Ass(A_{\mathfrak{m}_i})=\Spec {A_{\mathfrak{m}_i}}\cap\Ass(A)$, by choice of $x$ we have that $x$ does not lie in the union of the ideals in $\Ass(A_{\mathfrak{m}_i})$ and, thus, it is a regular element in $A_{\mathfrak{m}_i}$. This shows that also $I_{\mathfrak{m}_i}$ is projective and, therefore, so is $I$.
It remains to see that the support of $\sigma$ is $V(\mathfrak{p})$. We have the following cases.
\begin{itemize}
\item If $\mathfrak{q}$ does not lie in $V(x)$, then $\sigma_\mathfrak{q}$ is an isomorphism because $x$ is invertible in $A_\q$;
\item If $\mathfrak{q}$ lies in $V(x)\setminus V(\mathfrak{p})$, i.e. $\mathfrak{q}=\mathfrak{m_i}$ for some $i$, then $I=xA_{\mathfrak{m}_i}$ and, since as seen above $x$ is a regular element in $A_{\mathfrak{m}_i}$, it follows that $\sigma_{\mathfrak{m}_i}$ is an isomorphism;
\item If $\mathfrak{q}$ lies in $V(\mathfrak{p})$, then $I_{\mathfrak{q}}=A_{\mathfrak{q}}$ and $xA_{\mathfrak{q}}$ lies in $\mathfrak{p}A_{\mathfrak{q}}\subseteq \mathfrak{q}A_{\mathfrak{q}}$, thus showing that $\sigma_\mathfrak{q}$ is not surjective.
\end{itemize}
\end{proof}

Combining Theorem \ref{Krull dim one} with the results of Section \ref{Section ring epi}, we get the following immediate corollary.
\begin{corollary}\label{one}
If $A$ is a commutative noetherian ring of Krull dimension at most one, then there are bijections between
\begin{enumerate}
\item[(i)] epiclasses of universal localisations of $A$;
\item[(ii)] specialisation closed subsets of $\Spec A$;
\item[(iii)] smashing subcategories of $\Dcal(A)$.
\end{enumerate}
\end{corollary}

\begin{remark}
Note that the theorem above also provides a refinement of the telescope conjecture for rings of Krull dimension one. Indeed, the telescope conjecture is known to hold for commutative noetherian rings: all smashing subcategories are compactly generated by \cite{N}. Now we see  that in Krull dimension one the set of compact generators can be chosen to consist of 2-term complexes.
\end{remark}


\subsection{Normal rings}\label{normal}
We turn to a class of rings where we have additional information about the divisors and their groups.
Recall that a commutative noetherian ring $A$ is said to be \textbf{normal} if every localisation of $A$ at a prime is an integrally closed domain.

\begin{remark}\label{properties of a normal ring}
\hfill
\begin{enumerate}
\item Regular rings are locally factorial (the localisations at primes are unique factorisation domains by Auslander-Buchsbaum's theorem). Moreover, locally factorial rings are normal since every unique factorisation domain is integrally closed. None of the implications is reversible.
\item A finite product of normal rings is normal and every normal ring is a finite product of normal domains. If a ring $A$ is normal, then so are $A[X]$ and $A[[X]]$.
\item A ring of Krull dimension one is normal if and only if it is regular. More generally, if $A$ is normal, then for every prime $\p$ of height one, the ring $A_\mathfrak{p}$ is regular (so a discrete valuation domain).
\item If a ring $A$ is Cohen-Macaulay, then $A$ is normal if and only if $A_\p$ is regular for every prime $\mathfrak{p}$ of height one, see~\cite[Theorem 11.5]{E} and the discussion before it. In other words, a Cohen-Macaulay ring is normal if and only if it is regular in height one (the singular locus lives in height at least 2). Recall also that complete intersection rings and, more generally, Gorenstein rings are Cohen-Macaulay.
\end{enumerate}
\end{remark}

For commutative noetherian normal rings, we know more about the maps in \eqref{diagram of divisors}.

\begin{theorem}\cite[Theorem 11.8b and Proposition 11.11]{E}\label{relation groups} If $A$ is normal, then $\rm{div}$ and $\overline{\rm{div}}$ are injective. If, moreover, $A$ is locally factorial, then both $\rm{div}$ and $\overline{\rm{div}}$ are in fact isomorphisms.
\end{theorem}

Since every normal ring $A$ decomposes into a finite product $A_1\times A_2\times\cdots\times A_n$ of normal domains, $\ModA$ is consequently equivalent to $\ModA_1\times\ModA_2\times\cdots\times\ModA_n$ and $\Spec{A}$ is homeomorphic to the disjoint union $\Spec{A_1}\amalg\Spec{A_2}\amalg\cdots\amalg\Spec{A_n}$. A specialisation closed subset $V\subseteq\Spec{A}$ can thus be identified with an $n$-tuple $(V_1,V_2,\dots,V_n)$, where $V_i = V\cap\Spec{A_i}$ is specialisation closed in $\Spec{A_i}$.
Correspondingly, we have an isomorphism $\Pic{A}\cong\Pic{A_1}\times\Pic{A_2}\times\cdots\times\Pic{A_n}$ as abelian groups, and similar isomorphisms for $\C{A}$, $\Div{A}$ and $\Chow{A}$. Moreover, each ring epimorphism $f\colon A\longrightarrow B$ decomposes into a product of ring epimorphisms $f_i\colon A_i \longrightarrow B_i$. Clearly $f$ is flat if and only if so are all the $f_i$, and the same is true for the property of being a classical or a universal localisation. As an upshot, whenever we ask for some general criteria to compare flat ring epimorphisms with universal or even classical localisations of $A$, we can restrict ourselves to the case when $A$ is a domain.

For a (not necessarily normal) domain $A$ there is a rather easy description of all flat ring epimorphisms. These are certain intermediate ring epimorphisms between $A$ and its quotient field $K$. The following result essentially goes back to Richman~\cite{Rich}.

\begin{proposition}\label{flat epis as generalized quotients}
Let $A$ be a commutative noetherian domain and $f\colon A\longrightarrow B$ be a non-zero flat ring epimorphism corresponding to a specialisation closed subset $V$ of $\Spec{A}$. Then $f$ is equivalent to the inclusion
\[ A \la \bigcap_{\p\in\Spec{A}\setminus V} A_\p \qquad (\subseteq K). \]
If, moreover, $A$ is normal, then also $B$ is normal and $f$ is equivalent to
\[ A \la \bigcap_{\begin{smallmatrix}\p\in\Spec{A}\setminus V \\ \htt(\p) = 1\end{smallmatrix}} A_\p \qquad (\subseteq K). \]
\end{proposition}

\begin{proof}
It follows from \cite[Remark on p.~47]{S} that $f$ is equivalent to the inclusion map from $A$ to a subring $B'$ of $K$ which is isomorphic to $B$.
The formula $B'=\bigcap_{\p\in\Spec{A}\setminus V} A_\p$ follows from \cite[\S2]{Rich}.
If $A$ is normal, so is $B'$ by \cite[\S4]{Rich}. Then $B' = \bigcap_{\q\in\Spec {B'}, \htt(\q)=1} B'_\q$ by \cite[Corollary 11.4]{E}. However, we have that $B'_\q$ coincides with $A_\p$ (for  $\p = A \cap \q$) inside $K$. Moreover, the prime $\p=A \cap \q$ is of height one since the map $\q \mapsto \q \cap A$ identifies by \cite[Theorem 3]{Rich} (or Remark~\ref{quick way to see support}) the spectrum $\Spec{B'}$ with $\Spec{A}\setminus V$, which is closed under generalisation.
\end{proof}

Recall from Theorem~\ref{flat epis} that the minimal primes in the specialisation closed subset associated with a flat ring epimorphism  have height at most one.
Let us first discuss the case of a locally factorial ring $A$. Note that in this case, prime ideals of height one are projective. Indeed, given $\mathfrak{p}$ of height one and a prime ideal $\mathfrak{q}$ containing $\mathfrak{p}$, we have that $\p_\q$ is a height one prime in $A_\q$. Since $A_\mathfrak{q}$ is a unique factorisation domain, $\mathfrak{p}_\mathfrak{q}$ is a principal ideal and, thus, free as an $A_\mathfrak{q}$-module. On the other hand, if $\q$ is a prime ideal not containing $\p$, it is also clear that $\p_\q\cong A_\q$ is a free $A_\q$-module. Hence $\mathfrak{p}$ is projective as an $A$-module.

\begin{theorem}\label{locally factorial}
Let $A$ be a locally factorial commutative noetherian ring. Suppose that $V$ is a specialisation closed subset of $\Spec A$ such that all minimal primes of $V$ have height at most one.
Then there is a set of maps $\Sigma$ between finitely generated projective $A$-modules such that $\supp(\Sigma)=V$. More precisely, we can take $\Sigma=\{\mathfrak{p}\hookrightarrow A:\mathfrak{p}\in {\rm min\ } V\}$.

In particular, every flat ring epimorphism is a universal localisation. Moreover,  every universal localisation is a classical ring of fractions if and only if $\Pic{A}$ is torsion.
\end{theorem}

\begin{proof}
We assume without loss of generality that $A$ is a domain. Excluding the trivial case where $V=\Spec{A}$, we will assume that the minimal primes of $V$ have height precisely one.
Since $A$ is locally factorial, it follows that every such prime ideal is 
projective as an $A$-module.
Thus, the set of inclusion maps $\Sigma=\{\iota_\p\colon\mathfrak{p}{\longrightarrow} A:\mathfrak{p}\in {\rm min\ } V\}$ is indeed a set of maps between finitely generated projective $A$-modules and $\supp \Sigma=V$ since $\supp {\iota_\p}=V(\mathfrak{p})$.

It follows from Lemma~\ref{suppunivloc} that every flat ring epimorphism is a universal localisation. Since the if-part in the last statement is Theorem~\ref{Pic torsion}, it only remains  to show that $\Pic{A}$ is torsion if every universal localisation is a classical ring of fractions. Recall from Theorem \ref{relation groups} that $\Pic{A}\cong \Chow{A}$.  Let $\mathfrak{p}$ be a prime of height one. Since $A$ is locally factorial, $\mathfrak{p}$ is projective, and we may again consider the universal localisation at the inclusion map $\iota_\p\colon\p\longrightarrow A$. By assumption, $A_{\iota_\mathfrak{p}}$ is the localisation at a multiplicative set $S$, hence  $\bigcup_{s\in S} V(s)=\supp\iota_\p=V(\mathfrak{p})$ by Remark~\ref{suppclassloc-again}.
As $\V(\p)$ is irreducible, $\V(\p)=\V(s)$ for some non-zero (hence regular) element $s$ in $S$.
That is, $\p$ is the unique prime ideal of height one containing $s$.
It follows that $\di(s)=n\cdot\mathfrak p$ for some $n>0$ and we have $n\cdot\p=0$ in $\Chow A$.
\end{proof}

Note that for an arbitrary normal domain $A$, we cannot argue like in the theorem above, since a prime $\p$ of height one in $\Spec A$ will not always be projective. We can, however, view $\p$ as an element in ${\rm Div}(A)$ and we can ask if $\p$ (or finite sums of copies of $\p$) can be written as ${\rm div}(I)$ for an invertible ideal $I$ of $A$. This amounts to  asking whether $\p$ is a torsion element in $\Chow A/\Pic A$ rather than in $\Pic{A}$.

Let us return to the pullback diagram of abelian groups from Section \ref{commutative algebra}
\begin{equation}\label{pullback divisors}
\vcenter{
\xymatrix{0\ar[r]&\{sA:s\in K^\times/A^\times\}\ar[r]\ar[d]^{\cong}&\C{A}\ar[r]\ar[d]^{\rm{div}}&\Pic{A}\ar[d]^{\overline{\rm{div}}}\ar[r]&0 \\ 0\ar[r]& {\rm PDiv(A)}\ar[r] & \Div{A}\ar[r]&\Chow{A}\ar[r]&0}
}
\end{equation}
where the maps ${\rm div}$ and $\overline{\rm{div}}$ are injective.
The injectivity of $\di$ allows to reconstruct an invertible fractional ideal from its Weil divisor. The following lemma shows an explicit formula for such a reconstruction. We will use that if $\p$ is a prime ideal of $A$ of height one, then the normality of $A$ implies that $A_\p$ is a discrete valuation domain (Remark \ref{properties of a normal ring}(3)) and, thus, $\p A_\p$ is an invertible ideal of $A_\p$.

\begin{lemma}\label{reconstruct invertible ideal}
Let $A$ be normal and $I$ be an invertible fractional ideal with $\di(I) = \sum_{\p} n_\p\cdot \p$. Then $I = \bigcap_{\p} (\p A_\p)^{n_\p}$, where $\p$ runs over all prime ideals of height one.
If, moreover, $\di(I)$ is effective, then $I$ is an invertible ideal of $A$.
\end{lemma}

\begin{proof}
As explained above, we can assume without loss of generality that $A$ is a domain. We start with the case where $\di(I)$ is effective. By assumption, there is a regular element $g$ in $A$ such that $gI$ is an invertible ideal of $A$ and we have $\di(I)=\di(gI)-\di(gA)$.
For any prime $\p$ in $\Spec A$ of height one, it follows that $\ell(A_\p/gI_\p)\geq \ell(A_\p/gA_\p)$. Since, by assumption, $A_\p$ is a discrete valuation ring, this further implies that $gI_\p\subseteq gA_\p$ and, hence, $I_\p$ is contained in $A_\p$. As a consequence, $I$ is contained in $\bigcap_{\p\in\Spec{A},\htt(\p)=1} A_\p$ which coincides with $A$ by \cite[Corollary 11.4]{E}. 

Now we write $\di(I) = n_1\cdot\p_1+n_2\cdot\p_2+\cdots+n_s\cdot\p_s$ with $\p_1,\dots,\p_s$ pairwise distinct and $n_i\ge 0$ for all $i=1,\dots,s$.
We claim that the prime ideals associated to $A/I$ are all among $\p_1,\dots,\p_s$. Indeed, if $\p$ is associated to $A/I$, then $\p A_\p$ is associated to $A_\p/I_\p$. However, since $I$ is invertible, $I_\p$ is principal and generated by a regular element and, consequently, $\p$ is of height one by~\cite[Theorem 11.5(i)]{E}. Thus $\p=\p_i$ for some $i=1,\dots,s$ by the very definition of the map $\di$. This proves the claim.

Since $I\subseteq I_{\p_i} = (\p_iA_{\p_i})^{n_i}$, we clearly have $I \subseteq J$, where we put $J=A \cap \bigcap_{i=1}^s (\p_i A_{\p_i})^{n_i}$. Moreover, since localisation commutes with finite intersections by left exactness, we have $I_{\p}= J_{\p}$ for each $\p$ of height one. In particular, $\supp{J/I}$ does not intersect the associated primes of $A/I$, which implies that $J/I=0$ by~\cite[Corollary 3.5(b)]{E}. Using the equality $A=\bigcap_{\p\in\Spec{A},\htt(\p)=1} A_\p$ again, we obtain the desired expression $I = \bigcap_{\p\in\Spec{A},\htt(\p)=1} (\p A_\p)^{n_\p}$.

Suppose finally that $I$ is an arbitrary invertible fractional ideal and $g$ in $A$ is such that $gI \subseteq A$. If we write $\di(gA) = \sum_{\p} m_\p\cdot \p$, then $gI = \bigcap_{\p} (\p A_\p)^{m_\p+n_\p}$. Now it suffices to notice that the multiplication by $g^{-1}$ induces an $A$-module automorphism of the total ring of fractions of $A$ which maps $gI$ to $I$ and $(\p A_\p)^{m_\p+n_\p}$ to $(\p A_\p)^{n_\p}$.
\end{proof}

\begin{remark}\label{remark reconstruct invertible ideal}
If $A$ is normal, $I\subseteq A$ is an invertible ideal and $\di(I) = n_1\cdot\p_1+n_2\cdot\p_2+\cdots+n_s\cdot\p_s$, then $I=\bigcap_{i=1}^s \p_i^{(n_i)}$, where $\p_i^{(n_i)}$ is the \textbf{$n_i$-th symbolic power} of $\p_i$, \cite[\S3.9]{E}. Indeed, recall that $\p_i^{(n_i)}$ is defined as the $\p_i$-primary component in the primary decomposition of $\p^{n_i}$ and, by~\cite[\S3.3]{E}, we have $\p_i^{(n_i)} = A\cap(\p_iA_{\p_i})^{n_i}$. We have seen in the proof of Lemma~\ref{reconstruct invertible ideal} that $I = A \cap \bigcap_{i=1}^s (\p_i A_{\p_i})^{n_i} = \bigcap_{i=1}^s (A\cap(\p_i A_{\p_i})^{n_i})$, which gives the desired equality.
\end{remark}

The next example shows that, without the assumption of normality, even the last statement in the lemma above may fail. 

\begin{example}
Consider the ring $A=\mathbb{C}[X,Y]/(Y^2-X^3)$ and its invertible fractional ideal $I$ generated by $Y/X$. We claim that ${\rm div}(I)$ is effective. Since ${\rm div}(I)={\rm div}(YA)-{\rm div}(XA)$, it is enough to check that $\ell(A_\p/YA_\p)\geq \ell(A_\p/XA_\p)$ for all non-zero prime ideals $\p$ in $\Spec A$. But $\ell(A_\p/XA_\p)\not= 0$ is only possible if $X$ lies in $\mathfrak{p}$ or, equivalently, if $\p$ is the prime ideal of $A$ generated by $X$ and $Y$. But in this case, we have $3=\ell(A_\p/YA_\p)>\ell(A_\p/XA_\p)=2$. Therefore, ${\rm div}(I)$ is effective even though $I$ is not an ideal of $A$.
\end{example}

We are now ready to prove the main result of this subsection.

\begin{theorem}\label{thm normal domain}
Let $A$ be a normal commutative noetherian ring and $V$ be a specialisation closed subset of $\Spec A$ such that all minimal primes of $V$ have height at most one. Then the following holds.
\begin{itemize}
\item If all minimal primes $\p$ in $V$ which are of height exactly one are torsion in $\Chow A/\Pic A$, then there exists a universal localisation $f\colon A\longrightarrow B$ with associated specialisation closed subset $V$. Conversely, if $\p$ is a prime ideal of height one and $f\colon A\longrightarrow B$ is a universal localisation with associated specialisation closed subset $V=V(\p)$, then $\p$ is torsion in $\Chow A/\Pic A$.
\item If all minimal primes $\p$ in $V$ which are of height one are torsion in $\Chow A$, then $f$ as above is a classical localisation. Conversely, if $V=V(\p)$ and $f$ is a classical localisation, $\p$ is torsion in $\Chow A$.
\end{itemize}
\end{theorem}

\begin{proof}
We assume without loss of generality that $A$ is a domain and $0$ does not lie in $V$.
In order to show that there exists a universal localisation $f\colon A\longrightarrow B$ corresponding to $V$ it suffices to find for each minimal prime $\mathfrak{p}$ in $V$ a map $\sigma$ between finitely generated projective $A$-modules such that $\supp \sigma=V(\mathfrak{p})$. Let $\p$ be such a prime and suppose that $\p$ is torsion when seen as an element in $\Chow A/\Pic A$. Then there is an invertible fractional ideal $I$ of $A$ whose isoclass is mapped via $\overline{\rm{div}}$ to $n\cdot\p$ in $\Chow A$ for some $n>0$. Using the pullback diagram \eqref{pullback divisors}, it follows that we can choose $I$ in a way such that also ${\rm div}(I)$ equals $n\cdot\p$ in $\Div A$. By Lemma \ref{reconstruct invertible ideal}, this implies that $I$ is an ideal of $A$. Define $\sigma$ to be the inclusion $I\longrightarrow A$. It follows from Remark \ref{div and support} that $\supp \sigma=V(\mathfrak{p})$. Note that if $\p$ is even torsion when seen as an element in $\Chow A$, then the isoclass of $I$ is torsion in $\Pic A$, that is, there is some $m{>0}$ such that $I^m=sA$ is a principal ideal of $A$. Instead of localising at $\sigma$ we can now equivalently localise at the map $A\overset{s}{\longrightarrow}A$ turning $f$ into a classical localisation.

Conversely, suppose that $V=V(\p)$ and that $f$ is the universal localisation of $A$ at a set $\Sigma$ of maps between finitely generated projective $A$-modules whose associated specialisation closed subset is $V$. By Lemma \ref{suppunivloc}, we have $V(\p)=\supp\Sigma$. Since $V(\p)$ is irreducible, there is some map $\sigma$ in $\Sigma$ whose support is precisely $V(\p)$. By Proposition \ref{reduction}, we can assume that $\sigma$ is just the inclusion map of an invertible ideal $I$ into $A$. It then follows that ${\rm div}(I)$ is of the form $n\cdot\p$ for some $n\in\mathbb{N}$, which clearly implies that $\p$ is torsion when seen as an element in $\Chow A/\Pic A$. Note that in case $f$ is even a classical localisation, $I$ is a principal ideal of $A$ and ${\rm div}(I)=n\cdot\p$ becomes zero in $\Chow A$ or, in other words, $\p$ is torsion when seen as an element in $\Chow A$.
\end{proof}

\begin{remark}
When $A$ is locally factorial, the quotient $\Chow A/\Pic A$ will always be zero, and we recover  Theorem \ref{locally factorial}.
\end{remark}

Given a prime $\p$ of height one over a normal domain $A$, we will see in the forthcoming section that the specialisation closed subset $V(\p)$ does not necessarily arise from a flat ring epimorphism $A\longrightarrow B$ or, equivalently, its complement in $\Spec A$ is not necessarily coherent. So far we only know from Theorem \ref{thm normal domain} that whenever $\p$ is torsion when seen as an element in $\Chow A/\Pic A$, then $V(\p)$ must be of the form $\supp \sigma$ for a map $\sigma$ between finitely generated projective $A$-modules. In particular, $V(\p)$ is then associated to a universal localisation of $A$. In order to discuss examples of flat ring epimorphisms that are not universal localisations, we need to find some different criteria to guarantee that a subset of the form $V(\p)$ admits coherent complement in $\Spec A$. We will also deal with this in the next section.


\section{Examples}\label{section examples}

In this section we discuss some examples of noetherian normal domains $A$ and
\begin{enumerate}
\item primes $\p$ of height one such that the complement of $V(\p)$ is not coherent (i.e.\ there is no flat ring epimorphism $A\longrightarrow B$ associated with $V(\p)$), and
\item primes $\p$ of height one such that there is a flat ring epimorphism $A\longrightarrow B$ associated with $V(\p)$, but it is not a universal localisation.
\end{enumerate}

By Theorems~\ref{Krull dim one} and~\ref{locally factorial}, such examples must be at least $2$-dimensional and cannot be locally factorial. In the case where $A$ has Krull dimension exactly two, the singular locus must consist of maximal ideals of height two by normality.

Along the way, we obtain a useful criterion for the existence of flat ring epimorphisms in the case of normal rings of Krull dimension two (Proposition~\ref{prop criterion for coherence HLVT}).

\subsection{Closed subsets with non-coherent complement}

We start out with the first type of examples, i.e.\ primes $\p$ of height one such that the complement of $\V(\p)$ is not coherent. We have, in fact, already seen this phenomenon in Example~\ref{example two planes}, although in that case $A$ was not a domain. We provide a first easy example in a domain of Krull dimension three (a more comprehensive analysis of flat epimorphisms originating from a non-complete version of this ring will be given in Proposition~\ref{proposition 2x2 matrices}).

\begin{example} \label{example non-coh: dimension 3} \cite[Exercise 33]{Hu}
Let $A=k[[X,Y,U,V]]/(XU-YV)$ and $\p = (X,Y)$. Clearly, $A$ is a three-dimensional domain and, as $A/\p \cong k[[U,V]]$, $\p$ is a prime ideal of height one. Moreover, $A$ is normal by Remark~\ref{properties of a normal ring}(4) since it is a complete intersection ring and the singular locus of $A$ consists only of the maximal ideal $\m=(X,Y,U,V)$. We claim that $H^2_\p(A)\ne0$ and, hence, $V(\p)$ does not have coherent complement by Theorem~\ref{flat epis}. Indeed, notice first that $H^3_\p\equiv 0$ since $\p$ has two generators, \cite[Theorem 3.3.1]{BS2}. In particular, $H^2_\p$ is right exact and it suffices to show that $H^2_\p(R)\ne0$, where $R= A/(V) \cong k[[X,Y,U]]/(XU)$. However, $H^2_\p(R) \cong H^2_{\p R}(R)$ by the independence of the base (Proposition~\ref{properties of H_I}(1)), and the latter local cohomology module is non-zero by Example \ref{example two planes}.
\end{example}

There also exist examples of Krull dimension two, but they are more demanding as it turns out that they cannot be localisations of finitely generated algebras over a field. We will explain the reason here as this discussion is also necessary for \S\ref{subsec flat but not loc}.

Recall from \cite[\S32]{Mat} that a commutative noetherian ring $A$ is a \textbf{G-ring}, if the completion morphism $A_\p \longrightarrow \widehat{A_\p}$ is a regular map for each $\p\in\Spec A$. We will not discuss the definition more in detail here, but we rather point out the following classes of G-rings.

\begin{lemma} \label{lemma G-rings}
Finitely generated commutative algebras over a field as well as complete local commutative noetherian rings are G-rings. Furthermore, the class of G-rings is closed under flat epimorphic images (in particular, under classical localisations).
\end{lemma}

\begin{proof}
The first sentence follows from \cite[Theorem 32.3]{Mat} and Corollary to \cite[Theorem 32.6]{Mat}. If $A\longrightarrow B$ is a flat epimorphism, $\q\in\Spec B$ and $\p$ the preimage of $\q$, then $A_\p\cong B_\q$ by \cite[Proposition 2.4(iii)]{L}. Since G-rings are defined in terms of stalks, it follows that $B$ is a G-ring provided that $A$ is such.
\end{proof}

Recall further that a commutative noetherian local ring $A$ is called {\bf analytically irreducible} if its completion $\widehat A$ is a domain. If $A$ is regular, it is analytically irreducible since the completion $\widehat A$ is regular local and so, in particular, a (unique factorisation) domain.
Normality of $A$, in general, does not ensure analytic irreducibility, but it does for G-rings.

\begin{lemma} \label{lemma analytic irred}
Let $A$ be a local normal G-ring. Then $A$ is analytically irreducible.
\end{lemma}

\begin{proof}
This goes back to \cite{Zar}, and in this generality it follows from \cite[Theorems 32.2(1)]{Mat}.
\end{proof}

To see the connection to flat epimorphisms, notice that the analytic irreducibility of a commutative noetherian local ring $A$ of Krull dimension $d$ ensures that the only prime ideal $\q\in\Spec{\widehat A}$ for which $\dim \widehat A/\q=d$ is $\q=0$. Hence, by HLVT (Theorem \ref{HLVT'}), the local cohomology $H^d_I(A)$ vanishes for each proper ideal $I\subseteq A$ which is not $\m$-primary. This allows us to deduce the following criterion for the existence of flat epimorphisms with given support originating at rings of Krull dimension two.

\begin{proposition} \label{prop criterion for coherence HLVT}
Let $A$ be a commutative noetherian ring of Krull dimension two.

\begin{enumerate}
\item If $\p$ is a prime ideal in $\Spec A$ of height at most one such that for all $\m\in V(\p)$ of height two, the localisation $A_\m$ is analytically irreducible, then $V(\p)$ has coherent complement in $\Spec A$.
\item If $A$ is a normal G-ring and $V\subseteq\Spec A$ is any specialisation closed set such that the minimal primes in $V$ are of height at most one, then $V$ has coherent complement.
\end{enumerate}
\end{proposition}

\begin{proof}
(1) By Theorem~\ref{flat epis}, it is enough to show that $H_\p^{2}(A)=0$ or, equivalently, that $H_\p^{2}(A)_\m=0$ for all maximal ideals $\m\subseteq A$. Recall also that we have $H_\p^{2}(A)_\m\cong H_{\p A_\m}^{2}(A_\m)$ by the flat base change (Proposition~\ref{properties of H_I}(2)).
Now, we certainly have $H_{\p A_\m}^{2}(A_\m)=0$ if $\p$ is not contained in $\m$ (since then $\p A_\m = A_\m$) or if the height of $\m$ is at most one (by Theorem \ref{Gr vanishing}). If $\m$ is of height two and $\p\subseteq\m$, then $H_{\p A_\m}^{2}(A_\m) = 0$ directly by the discussion before the proposition.

(2) This follows immediately from \cite[Proposition 4.1(1)]{K}, the first part and Lemma~\ref{lemma analytic irred}.
\end{proof}

Thus, in order to exhibit an example of a local normal domain $A$ of dimension two and $\p\in\Spec A$ of height one such that the complement of $V(\p)$ is not coherent, we need that $A$ is not analytically irreducible, so in particular not a G-ring. An example of such a ring was given by Nagata \cite{Nag}. 

\begin{example} \label{example non-coh: dimension 2}
Let $k$ be a field of characteristic different from two, and let $X,Y$ and $Z$ be algebraically independent elements over $k$. Let $w=\sum_{i>0}a_iX^i\in k[[X]]$ be a transcendental element over $k(X)$ and set $Z_1=Z$ and $Z_{i+1}=X^{-i}(Z-(Y+\sum_{j<i}a_jX^j)^2)\in k(X)[Y,Z]$ for all $i\geq 1$. Consider $R=k[X,Y,Z_1,Z_2,\dots]_\m$, where $\m={(X,Y,Z_1,Z_2,\dots)}$, and $A=R[W]/(W^2-Z)$, where $W$ is a new variable. By \cite{Nag}, $A$ is a noetherian normal local domain of Krull dimension two, and $\widehat{A}=k[[X,Y]][W]/\big((W-(Y+w))(W+(Y+w))\big)$.

Put $\p= (X,W+Y)\in\Spec{A}$ and $\q=(W-(Y+w)) \in\Spec{\widehat A}$. We claim that $\p$ is of height one. To see that, note first that by the arguments in~\cite[part (2)]{Nag}, $R/(X)\cong k[Y]$. The relation $Z_2X=Z-Y^2$ in $R$ implies that the coset of $Z$ is sent to $Y^2$ under this isomorphism. Hence $A/(X) \cong k[Y,W]/(W^2-Y^2) = k[Y,W]/\big((W-Y)(W+Y)\big)$ and $A/\p \cong k[Y,W]/(W+Y)\cong k[Y]$ has dimension one, proving the claim.

Now $\dim\widehat{A}/\q=\dim\widehat{A}=2$ and $\p\widehat{A}+\q = (X,W+Y,W-Y+w) = (X,Y,W)$ since $w\in(X)$ in $\widehat A$ and $\operatorname{char}k\ne2$. Hence $\dim \widehat{A}/(\p\widehat{A}+\q)=0$ and Theorem \ref{HLVT'} implies $H_\p^2(A)\ne0$. Finally, by Theorem~\ref{flat epis}, the complement of $V(\p)$ in $\Spec{A}$ is not coherent.
\end{example}

\subsection{Groups of divisors: the graded and the projective case}
\label{subsection divisors advanced}

In order to give examples of flat epimorphisms which are not universal localisations, we need to compute the Picard and the divisor class groups for certain rings of Krull dimension two and three. We will focus on homogeneous coordinate rings of one and two-dimensional projective varieties where the projective version of the divisor class groups is well-known and extract the groups  we need from that piece of information. This subsection is devoted to describing the necessary relations between the various types of groups of divisors.

Suppose that $A=\bigoplus_{i\in\Z}A_i$ is a $\Z$-graded commutative noetherian normal ring and denote by $L = A_S$ the localisation at the multiplicative set $S\subseteq A$ consisting of all homogeneous non-zero-divisors. The finitely generated submodules $I\subseteq L$ will be called \textbf{homogeneous fractional ideals}. We will denote by $\HC{A}$ the subgroup of $\C{A}$ consisting of the homogeneous invertible fractional ideals, and by $\HPic{A}$ the \textbf{homogeneous Picard group}, i.e.\ the quotient of $\HC{A}$ by the subgroup of principal homogeneous invertible fractional ideals. In the same vein, we will denote by $\HDiv{A} \subseteq \Div{A}$ the free subgroup generated by the homogeneous ideals. We also define the subgroup $\HPDiv{A} \subseteq \PDiv{A}$ consisting of the divisors of the homogeneous elements of $L$. We will call the quotient $\HCl{A} = \HDiv{A}/\HPDiv{A}$ the \textbf{homogeneous divisor class group} of $A$.

Since associated prime ideals of graded $A$-modules are homogeneous by~\cite[\S3.5]{E}, the map $\di\colon\C{A}\longrightarrow\Div{A}$ restricts to $\di\colon\HC{A}\longrightarrow\HDiv{A}$ and we have a commutative diagram
\[
\xymatrix{
0\ar[r]&\{sA:s\in L^\times\textrm{ homogeneous}\}\ar[r]\ar[d]^{\cong}&\HC{A}\ar[r]\ar[d]^{\di}&\HPic{A}\ar[d]^{\overline{\di}}\ar[r]&0 \\
0\ar[r]& \HPDiv{A}\ar[r] & \HDiv{A}\ar[r]&\HCl{A}\ar[r]&0,
}
\]
where the maps $\di$ and $\overline{\di}$ are again injective. The following lemma is the first step of relating this diagram to the similar diagram~\eqref{diagram of divisors} for not necessarily homogeneous divisors.

\begin{lemma}\label{homogeneity by numerology}
Let $A=\bigoplus_{i\in\Z}A_i$ be a $\Z$-graded commutative noetherian normal ring, and let $I$ be an invertible fractional ideal of $A$. If $\di(I)$ lies in $\HDiv{A}$, then $I$ lies in $\HC{A}$. In particular, we have $\HPDiv{A}=\HDiv{A}\cap\PDiv{A}$.
\end{lemma}

\begin{proof}
If $I\subseteq A$, this is an immediate consequence of Remark~\ref{remark reconstruct invertible ideal} and the fact that $\p^{(n)}$ is homogeneous whenever $\p$ is homogeneous and $n\ge 0$ by~\cite[\S3.5]{E}. If $I$ is any invertible fractional ideal of $A$, this follows from the fact that there exists a homogeneous element $g\in A$ such that $\di(I)+\di(gA)$ is effective and, hence, $gI\subseteq A$ by Lemma~\ref{reconstruct invertible ideal}.
\end{proof}

The main result about homogeneous class groups now says that they are canonically isomorphic to the non-homogeneous ones for normal rings.

\begin{proposition} \label{homogeneous class groups main}
Let $A=\bigoplus_{i\in\Z}A_i$ be a $\Z$-graded commutative noetherian normal ring. The inclusions $\HC{A}{\longrightarrow}\C{A}$ and $\HDiv{A}{\longrightarrow}\Div{A}$ induce isomorphisms
$\HPic{A}\cong\Pic{A}$ and $\HCl{A}\cong\Chow{A}$.
\end{proposition}

\begin{proof}
This was proved in~\cite[Proposition 7.1]{Sam}, but we recall the proof for the reader's convenience. We  suppose without loss of generality that $A$ is a domain and, by excluding a trivial case, also that $A\ne A_0$. 

If we denote by $S\subseteq A$ the set of non-zero homogeneous elements and by $L=A_S$ the localisation at $S$, then $L\cong{K}[T^{\pm1}]$, where $K=L_0$ is a field and $T$ is any fixed homogeneous element of the smallest possible positive degree. Let $\p$ be a prime ideal of $A$ of height one, but not necessarily homogeneous. Since $L$ is a principal ideal domain, there exists an element $f$ in $\p$ such that $\p L = {fL}$. This implies that there exists $g$ in $S$ such that already $\p A_g = f A_g$. Hence the divisor $D := \p - \di(f)$ is supported in $V(g)$ and, consequently, $D$ lies in $\HDiv{A}$ and $[D] = [\p]$ in $\Chow{A}$. This implies that the induced map $\HCl{A}\longrightarrow\Chow{A}$ is surjective and the injectivity follows from Lemma~\ref{homogeneity by numerology}.
We thus have a commutative square with monomorphisms in the columns
\[
\xymatrix{
\HPic{A} \ar[r]\ar[d]_-{\overline{\di}} & \Pic{A}\ar[d]^-{\overline{\di}} \\
\HCl{A} \ar[r]^-{\cong} & \Cl{A}
}
\]
which is a pull-back by Lemma~\ref{homogeneity by numerology}, showing that also $\HPic{A}\longrightarrow\Pic{A}$ is an isomorphism. 
\end{proof}

If $k$ is a field and $A=\bigoplus_{i\ge0}A_i$ is non-negatively graded with $A_0=k$, we can form the projective scheme $\ProjSch{A}$. As a set, $X=\ProjSch{A}$ is the set of homogeneous prime ideals of $A$ different from the irrelevant ideal $\m=\bigoplus_{i\ge1}A_i$. To $X$ one can again attach its group of divisors $\Div{X}$ and its class group $\Chow{X}$. Divisors are again formal $\Z$-linear combinations of height one subvarieties of $X$, i.e.\ $\Div{X}=\HDiv{A}$. In the class group one factors out the subgroup of divisors of degree zero elements of $L=A_S$, where $S$ is the set of homogeneous regular elements of $A$. This allows us to quickly deduce the relation between $\Chow{A}$ and $\Chow{X}$.

\begin{proposition} \label{class group via projective scheme}
Let $k$ be a field, $A=\bigoplus_{i\ge0}A_i$ a graded normal domain of Krull dimension at least one with $A_0=k$ and $X=\ProjSch{A}$ its associated projective scheme. Let also $L$ be the ring of fractions obtained by localising $A$ at the set of non-zero homogeneous elements. For any homogeneous element $T$ in $L$ of  smallest possible positive degree, there is a short exact sequence
\[ 0\longrightarrow \Z \overset{\varphi}\longrightarrow \Chow{X} \longrightarrow \Chow{A} \longrightarrow 0. \]
where $\varphi(1)$ is the class of $\di(T)$.
\end{proposition}

\begin{proof}
This is in fact a special case of~\cite[Theorem 1.6]{Wat}. As in the proof of Proposition~\ref{homogeneous class groups main}, we have $L\cong K[T^{\pm1}]$, where $K=L_0$ is a field. Then we have a short exact sequence
\[ 0\longrightarrow \{\di(gA):g\in K\} \longrightarrow \HPDiv{A} \longrightarrow \Z \longrightarrow 0, \]
where the quotient $\Z$ is generated by the coset of $T$. Now we have $\Chow{A} \cong \HDiv{A}/\HPDiv{A}$ by Proposition~\ref{homogeneous class groups main}, whereas $\Chow{X} = \HDiv{A}/\{\di(gA):g\in K\}$.
\end{proof}

\begin{remark} \label{Kurano's iso}
Let $A=\bigoplus_{i\ge0}A_i$ be a graded commutative noetherian normal ring with $A_0=k$ and let $\m=\bigoplus_{i\ge1}A_i$ be the unique maximal graded ideal. Kurano established in \cite[Lemma (4.1)]{Kura} an isomorphism $\Chow{A}\cong\Chow{A_\m}$. Since every projective graded $A$-module is free, we have $\Pic{A} = 0$ by Proposition~\ref{homogeneous class groups main} and, hence, we also have an isomorphism $\Pic{A}\cong\Pic{A_\m}\;(=0)$. 
\end{remark}

\subsection{Flat epimorphisms which are not localisations}
\label{subsec flat but not loc}

If $A=\bigoplus_{i\ge0}A_i$ is the homogeneous coordinate ring of a $(d-1)$-dimensional smooth projective variety $X=\Proj A$ over a field $k$, then the singular locus of $\Spec{A}$ (which is geometrically the affine cone of $X$) contains at most the irrelevant ideal $\m=\bigoplus_{i\ge1}A_i$. If $d=2$ (i.e. if $X$ is a smooth projective curve) and $A$ is normal (which is automatic if $A$ is Cohen-Macaulay by Remark~\ref{properties of a normal ring}(4)), then any specialisation closed set $V\subseteq\Spec{A}$ with minimal primes of height at most one has coherent complement by Proposition~\ref{prop criterion for coherence HLVT} (together with Lemma~\ref{lemma G-rings}). In view of Theorem~\ref{thm normal domain}, all we need to do to give two-dimensional examples of flat epimorpshisms which are not universal localisations is to find such $A$ with non-torsion $\Chow{A}/\Pic{A}$. We moreover know that $\Pic{A}=0$ by Remark~\ref{Kurano's iso}.

Now we can give particular examples. Recall that an \textbf{elliptic curve} $E$ (embedded into the projective plane $\PP_k^2 = \ProjSch{k[X,Y,Z]}$ over a field $k$) is defined by $E=\ProjSch{A}$, where $A=k[X,Y,Z]/(f)$ is a standard graded ring with $f$ a homogeneous polynomial of degree $3$ such that $A_\p$ is a regular local ring for all homogeneous prime ideals $\p$ not containing (or, equivalently, not equal to) the irrelevant ideal $(X,Y,Z)$. In particular, $E$ is a smooth projective curve.

We will also assume for simplicity that the set $E(k)$ of $k$-rational points of $E$ contains an inflection point $\mathcal{O}$. Then $E(k)$ has a well-known group structure with the neutral element $\mathcal{O}$ and the group operation given by the rule that if a line in $\PP_k^2$ intersects $E(k)$ in three points (when counting the intersection with multiplicities), then the sum of the three points is $\mathcal{O}$. The divisor class group $\Chow{E}$ is well-known to be isomorphic to $E(k)\times\Z$ in this case. More specifically, there is a well-defined group homomorphism $\deg\colon \Chow{E}\longrightarrow\Z$ which sends the class of a closed point $\p$ of $E$ to the degree $\deg([\p]) = \dim_kk(\p)$. The kernel of $\deg$ is then identified with $E(k)$ via a group isomorphism which sends $\p$ in $E(k)$ to $[\p]-[\mathcal{O}]\in\Cl{E}$. Now we can use Propositions~\ref{homogeneous class groups main} and~\ref{class group via projective scheme} to compute $\Chow{A}$. If $T = aX+bY+cZ$ is an equation for the tangent of $E$ at $\mathcal{O}$, then ${\di(T)} = 3\cdot\mathcal{O}$ and, hence
\begin{equation}\label{class group elliptic curves}
\Cl{A} \cong E(k)\times \Z/(3).
\end{equation}

We summarise the results about homogeneous coordinate rings of elliptic curves in the following theorem.

\begin{theorem}\label{ce}
Let $A$ be the homogeneous coordinate ring of an elliptic curve $E$ over $k$ with a $k$-rational inflection point $\mathcal{O}$. If $\p$ lies in $E(k)$ (in particular, $\p$ is a homogeneous prime ideal of $A$ of height one) and $\p$ is a non-torsion point of the group $E(k)$, then there is a flat ring epimorphism $f\colon A\longrightarrow B$ associated with $V(\p)$ and $f$ is not a universal localisation.
\end{theorem}

\begin{proof}
The flat ring epimorphism exists by Proposition~\ref{prop criterion for coherence HLVT} and Lemma~\ref{lemma G-rings}. It is not a universal localisation by Theorem~\ref{thm normal domain} and the isomorphism~\eqref{class group elliptic curves}.
\end{proof}

\begin{example}
Let $A=\QQ[X,Y,Z]/(X^3-Y^2Z-4Z^3)$ be a graded ring, with the grading defined by $\deg(X)=\deg(Y)=\deg(Z)=1$, and let $\p$ denote the prime ideal $(X-2Z,Y-2Z)$. Then the assumption of Theorem \ref{ce} is satisfied; see \cite[Example 1.10]{KKS}. Hence we have flat ring epimorphisms originating from $A$ which are not universal localisations.

If we, on the other hand, let $A=\QQ[X,Y,Z]/(X^3-XZ^2-Y^2Z)$ or $A=\QQ[X,Y,Z]/(X^3-Y^2Z+Z^3)$, then $\Chow{A}$ is torsion (\cite[Examples 1.8 and 1.9]{KKS}). In particular, every flat ring epimorphism originating from $A$ is a universal localisation.
\end{example}

We conclude the paper with an analysis of flat epimorphisms and universal localisations originating from the algebra $A=k[X,Y,U,V]/(XU-YV)$ for an algebraically closed field $k$. This is a three-dimensional normal domain (compare with Example~\ref{example non-coh: dimension 3}). It turns out that there are both primes $\p\in\Spec A$ of height one such that $V(\p)$ does not have coherent complement and primes $\p$ where $V(\p)$ has coherent complement, but the corresponding flat epimorphism is not a universal localisation.

In order to compute the divisor class and the Picard groups, we equip $A$ with the standard grading with all $X,Y,U,V$ in degree $1$ and put $\m=(X,Y,U,V)=\bigoplus_{i\ge 1}A_i$. Then $\Pic{A}=0$ by Remark~\ref{Kurano's iso} and to compute $\Chow{A}$, we will use Propositions~\ref{homogeneous class groups main} and \ref{class group via projective scheme}.

To this end, it is well-known that $X := \ProjSch{A}\cong \PP_k^1\times\PP_k^1$ via the Segre embedding and that $\Chow{X} \cong \Chow{\PP_k^1\times\PP_k^1}\cong\Z\times\Z$, \cite[Example II.6.6.1]{HarAG}.
In fact, we can be more specific about this isomorphism. Consider the embedding $A \longrightarrow k[S_0,S_1,T_0,T_1]$ of $k$-algebras given by
\[
X \mapsto S_0T_0, \quad
Y \mapsto S_1T_0, \quad
U \mapsto S_1T_1, \quad
V \mapsto S_0T_1.
\]
This is essentially the Segre embedding at the level of homogeneous coordinate rings, where we equip $k[S_0,S_1,T_0,T_1]$ with the $\Z\!\times\!\Z$-grading with $S_0,S_1$ in degree $(1,0)$ and $T_0,T_1$ in degree $(0,1)$. The isomorphism of schemes $X\cong\PP_k^1\times\PP_k^1$ in particular induces a bijection between homogeneous prime ideals $\p\subseteq A$ of height one and homogeneous prime ideals $\p'\subseteq k[S_0,S_1,T_0,T_1]$ of height one, which is given by $\p=\p'\cap A$. Since every height one prime ideal of a polynomial ring is principal, the following fact immediately follows from the discussion.

\begin{lemma}
For any homogeneous prime ideal $\p\subseteq A=k[X,Y,U,V]/(XU-YV)$, there is a homogeneous (with respect to the $\Z\!\times\!\Z$-grading) irreducible polynomial $f_\p\in k[S_0,S_1,T_0,T_1]$ such that $\p = (f_\p\cdot k[S_0,S_1,T_0,T_1])\cap A$. Moreover, such $f_\p$ is unique up to a non-zero scalar multiple.
\end{lemma}

Now, the bijection $\psi\colon\Chow{X}\longrightarrow \Z\times\Z$ is explicitly given as follows. If $[\p]\in\Chow{X}$ is a divisor class represented by a homogeneous prime ideal $\p\subseteq A$, $f_\p$ is the homogeneous polynomial as in the above lemma, $d_\p$ is the degree of $f_\p$ in the variables $S_0,S_1$ and $e_\p$ is the degree of $f_\p$ in $T_0,T_1$, then we have
\begin{align*}
\psi\colon \Chow{\ProjSch{A}} &\overset{\cong}\longrightarrow \Z\times\Z, \\
[\p] &\longmapsto (d_\p,e_\p).
\end{align*}
In particular, one readily checks that
\[
\psi([(X,V)]) = \psi([(Y,U)]) = (1,0)
\quad\textrm{and}\quad
\psi([(X,Y)]) = \psi([(U,V)]) = (0,1),
\]
since $f_{(X,V)}=S_0$, $f_{(Y,U)}=S_1$, $f_{(X,Y)}=T_0$ and $f_{(U,V)}=T_1$. Since $(X,V)\cdot(X,Y) = X\cdot\m$ as ideals in $A$, we have that $\di(X)=(X,V)+(X,Y)$ in $\HDiv{A}$, which corresponds to $(1,1)\in\Z\times\Z$ under $\psi$ (see also~\cite[Example II.6.6.2]{HarAG}).
Now we can use Proposition~\ref{class group via projective scheme} with $T=X$ to conclude that $\Chow{A} \cong \Z\times\Z/\big((1,1)\big)\cong \Z$. With the notation above, if $\p$ is a homogeneous prime ideal of $A$ of height one and $d_\p$, $e_\p$ are the degrees of $f_\p$, we explicitly have
\begin{align*}
\rho\colon \Chow{A} &\overset{\cong}\longrightarrow \Z, \\
[\p] &\longmapsto e_\p - d_\p.
\end{align*}

The existence of flat epimorphisms and universal localisations starting in $A$, depending on the class groups, is summarised in the following proposition. For a particular instance of case \eqref{item example of flat non-univ loc}, we refer to \cite[Example 6.14]{MS} and \cite[Remarks 2.1]{HelSt}.

\begin{proposition} \label{proposition 2x2 matrices}
Let $k$ be an algebraically closed field and $A=k[X,Y,U,V]/(XU-YV)$ with the standard $\Z$-grading. If $\p\subseteq A$ is a homogeneous prime ideal of height one and $\psi([\p]) = (d_\p,e_\p)$ with the notation above, exactly one of the following three cases occurs.

\begin{enumerate}
\item If $d_\p=0$ or $e_\p=0$, then $V(\p)$ does not have coherent complement. This case occurs precisely if $\p$ is of the form $(g(X,Y),g(V,U))$ or $(g(X,V),g(Y,U))$, where $g$ is a linear form.
\item\label{item example of flat non-univ loc} If $d_\p\ne 0\ne e_\p$ and $d_\p\ne e_\p$, then $V(\p)$ has coherent complement, but the corresponding flat epimorphism $A\longrightarrow B$ is not a universal localisation.
\item If $d_\p=e_\p$, then $V(\p)$ corresponds to a classical localisation.
\end{enumerate}
\end{proposition}

\begin{proof}
Since clearly $(d_\p,e_\p)\ne(0,0)$, exactly one of the three cases occurs.

In case (1), assume without loss of generality that $e_\p=0$. Then $f_\p$ (in the notation above) is an irreducible homogeneous polynomial in $k[S_0,S_1]$ with the standard grading and $\p = (f_\p\cdot k[S_0,S_1,T_0,T_1])\cap A$. In particular, $f_\p(X,Y)$ and $f_\p(V,U)$ belong to $\p$. Since $k$ is algebraically closed, $f_\p$ must be linear and one readily checks that $\p=(f_\p(X,Y),f_\p(V,U))$.

Now, we can without loss of generality assume that $S_1$ does not divide $f_\p$, i.e.\ $f_\p(S_0,0)\ne0$. As in Example~\ref{example non-coh: dimension 3}, we have $H^i_\p(A)=0$ for $i\ge 3$ by \cite[Theorem 3.3.1]{BS2} since $\p$ has two generators.
To prove that $V(\p)$ does not have coherent complement, we need to show that $H^2_\p(A)\ne 0$ (Theorem~\ref{flat epis}). Since $H^2_\p$ is right exact, it again suffices to prove by the independence of the base (Proposition~\ref{properties of H_I}) that $H^2_{\p C}(C)\ne 0$ for $C = A/(Y) \cong k[X,U,V]/(XU)$. Note that $\q=(U)\subseteq\widehat C$ satisfies $\dim \widehat C = \dim \widehat C/\q = 2$ and, since $\widehat{C}/(\p\widehat C+\q)\cong K[X,V]/I$ for an ideal $I$ containing $f_\p(X,0)$ and $f_\p(V,0)$, also $\dim \widehat{C}/(\p\widehat C+\q) = 0$. The non-vanishing $H^2_{\p C}(C)$ then follows from Theorem~\ref{HLVT'}.

In case (2), assume without loss of generality that $d_\p<e_\p$ and put $\delta_\p:=e_\p-d_\p=\rho([\p])$, where $\rho\colon\Chow{A}\overset{\cong}\longrightarrow\Z$ is as above. We can equip $k[S_0,S_1,T_0,T_1]$ with a $\Z$-grading such that $S_0,S_1$ are in degree $-1$ and $T_0,T_1$ are in degree $1$. This means that $A$ identifies with $k[S_0,S_1,T_0,T_1]_0$ and $f_\p\in k[S_0,S_1,T_0,T_1]_{\delta_\p}$. Since then $f_\p S_0^{\delta_\p}, f_\p  S_1^{\delta_\p} \in \p$ and $f_\p^{2\delta_\p-1}\in(S_0^{\delta_\p}, S_1^{\delta_\p})$, we have $f_\p^{2\delta_\p} \in \p \cdot k[S_0,S_1,T_0,T_1]$. It follows from (the proof of) \cite[Proposition 6.11]{MS} that there exists a flat epimorphism $A\longrightarrow B$ associated with $V(\p)$. On the other hand $[\p]\in\Chow{A}\cong \Z$ corresponds to the torsion-free element $0\ne\delta_\p$, so it cannot be a universal localisation by Theorem~\ref{thm normal domain}.

Finally, in case (3), $f_\p$ lies in the image of the embedding $A \longrightarrow k[S_0,S_1,T_0,T_1]$ and, hence $\p=f_\p A$ is principal. The classical localisation $A\longrightarrow A_{f_\p}$ is associated with $V(\p)$ in this case.
\end{proof}


\section*{Acknowledgments}
The authors are grateful to Kazuhiko Kurano and Shunsuke Takagi for valuable information on divisor class groups and elliptic curves, and to  Changchang Xi and Joseph Chuang for valuable discussions.


\end{document}